\documentclass[a4paper]{amsart} 
\usepackage{amssymb,times, amscd,amsmath,amsthm, xypic}
\usepackage{tikz}
\usepackage{tikz-cd}
\usepackage[all]{xy}
\usepackage{geometry}
\usepackage{mathrsfs}

\begin{document} 

\numberwithin{equation}{section}
\newtheorem{thm}[equation]{Theorem}
\newtheorem{pro}[equation]{Proposition}
\newtheorem{prob}[equation]{Problem}
\newtheorem{qu}[equation]{Question}
\newtheorem{cor}[equation]{Corollary}
\newtheorem{con}[equation]{Conjecture}
\newtheorem{lem}[equation]{Lemma}
\theoremstyle{definition}
\newtheorem{ex}[equation]{Example}
\newtheorem{defn}[equation]{Definition}
\newtheorem{observation}[equation]{Observation}
\newtheorem{rem}[equation]{Remark}

\renewcommand{\rmdefault}{ptm}
\newcommand{\pdarrow}[2]{\ar@<0.5ex>[d]^-{#1} \ar@<-0.5ex>[d]_-{#2}}

\def \calB{\mathcal B}
\def\frak{\mathfrak}
\def\alp{\alpha}
\def\be{\beta}
\def\jeden{1\hskip-3.5pt1}
\def\om{\omega}
\def\bigstar{\mathbf{\star}}
\def\ep{\epsilon}
\def\vep{\varepsilon}
\def\Om{\Omega}
\def\la{\lambda}
\def\La{\Lambda}
\def\si{\sigma}
\def\Si{\Sigma}
\def\Cal{\mathcal}
\def\m {\mathcal}
\def\ga{\gamma}
\def\Ga{\Gamma}
\def\de{\delta}
\def\De{\Delta}
\def\bF{\mathbb{F}}
\def\bH{\mathbb H}
\def\bPH{\mathbb {PH}}
\def \bB{\mathbb B}
\def \bA{\mathbb A}
\def \bC{\mathbb C}
\def \bOB{\mathbb {OB}}
\def \bM{\mathbb M}
\def \bOM{\mathbb {OM}}
\def \mA{\mathcal A}
\def \mB{\mathcal B}
\def \mC{\mathcal C}
\def \mR{\mathcal R}
\def \mH{\mathcal H}
\def \mM{\mathcal M}
\def \mV{\mathcal V}
\def \mTOP{\mathcal {TOP}}
\def \mAB{\mathcal {AB}}
\def \bI{\mathbb I}
\def \bK{\mathbb K}
\def \bG{\mathbf G}
\def \bL{\mathbb L}
\def\bN{\mathbb N}
\def\bR{\mathbb R}
\def\bP{\mathbb P}
\def\bZ{\mathbb Z}
\def\bC{\mathbb  C}
\def \bQ{\mathbb Q}
\def\op{\operatorname}

\newcommand{\maru}[1]{\ooalign{
\hfil\resizebox{.8\width}{\height}{#1}\hfil
\crcr
\raise.1ex\hbox{\large$\bigcirc$}}}

\newcommand{\Maru}[1]{\ooalign{
\hfil\resizebox{.6\width}{\height}{#1}\hfil
\crcr
\raise.1ex\hbox{\LARGE$\bigcirc$}}}

\newcommand{\MMaru}[1]{\ooalign{
\hfil\resizebox{.5\width}{\height}{#1}\hfil
\crcr
\raise.1ex\hbox{\Huge$\bigcirc$}}}

\newcommand{\MMMaru}[1]{\ooalign{
\hfil\resizebox{.4\width}{\height}{#1}\hfil
\crcr
\raise.1ex\hbox{\Huge$\bigcirc$}}}

\newcommand{\MMMMaru}[1]{\ooalign{
\hfil\resizebox{.3\width}{\height}{#1}\hfil
\crcr
\raise.1ex\hbox{\Huge$\bigcirc$}}}

\newcommand{\MMMMMaru}[1]{\ooalign{
\hfil\resizebox{.2\width}{\height}{#1}\hfil
\crcr
\raise.1ex\hbox{\Huge$\bigcirc$}}}

\title{Co-operational bivariant theory
}

\thanks {
\noindent
\emph{keywords} : bivariant theory, operational bivariant theory, cohomology operation \\
\emph{Mathematics Subject Classification 2000}: 55N35, 55S99, 14F99}

\author{Shoji Yokura}

\date{}
\address{Graduate School of Science and Engineering, Kagoshima University, 1-21-35 Korimoto, Kagoshima, 890-0065, Japan}
\email{yokura@sci.kagoshima-u.ac.jp}

\begin{abstract}
For a covariant functor 
W. Fulton and R. MacPherson defined \emph{an operational bivariant theory} associated to this covariant functor. In this paper we will show that given a contravariant functor 
one can similarly construct a ``dual" version of an operational bivariant theory, which we call a \emph{co-operational} bivariant theory. 
If a given contravariant functor is the usual cohomology theory, then our co-operational bivariant group 
for the identity map consists of what are usually called ``cohomology operations". In this sense, our co-operational bivariant theory consists of \emph{``generalized"} cohomology operations.
\end{abstract}

\maketitle

\section{Introduction}
In \cite{FM} W. Fulton and R. MacPherson have introduced \emph{bivariant theory} $\mathbb B^*(X \xrightarrow f Y)$  with an aim to deal with Riemann--Roch type theorems for singular spaces and to unify them. $\mathbb B_*(X):= \mathbb B^{-*}(X \xrightarrow {} pt)$ becomes a covariant functor and $\mathbb B^*(X):= \mathbb B^{*}(X \xrightarrow {\op{id}_X} X)$  
 a contravariant functor. In this sense  $\mathbb B^*(X \xrightarrow f Y)$ is called a \emph{bivariant} theory\footnote{As recalled in \S 2, a bivariant theory $\mathbb B^*$ assigns to each map $f:X \to Y$ a graded abelian group (sometimes, a set, etc.) $\mathbb B^*(X \xrightarrow f Y)$. For the sake of simplicity, sometimes we write a bivariant theory $\mathbb B^*(X \to Y)$ or $\mathbb B(X \to Y)$ , instead of a bivariant theory $\mathbb B^*$ or $\mathbb B$  , unless some confusion is possible.}. 
In \S \ref{BT} below we make a quick recall of Fulton--MacPherson's bivariant theory. Here we just remark that in \cite[\S 2.2 Axioms for a bivariant theory]{FM} the value $\mathbb B^*(X \xrightarrow f Y)$ is in the category of graded abelian groups, but as remarked in \cite[Remark, p.22]{FM} $\mathbb B^*(X \xrightarrow f Y)$ can be valued \emph{in an arbitrary category}, e.g., the category of abelian groups (see \cite[\S 6.1.2 Definition of $\mathbb F$, p.60]{FM}), the category of sets (see \cite[\S 4.3 Differentiable Riemann--Roch]{FM}), the category of derived categories (see \cite[\S 7 Grothendieck Duality and Derived Functors]{FM}), etc.
 
As a typical example of such a bivariant theory Fulton and MacPherson defined a bivariant homology theory $\mathbb H^*(X \xrightarrow f Y)$ constructed from the usual cohomology theory $H^*$, which is a multiplicative cohomology theory. Here the multiplicativity is crucial. Given a homology theory, i.e., a covariant functor with values in the category of (graded) abelian groups, in \cite[\S 8.1]{FM} they defined what is called \emph{an operational bivariant theory} or simply \emph{an operational theory}. 
Here we note that in fact there are at least \emph{three more} kinds of operational bivariant theories requiring some other extra conditions, depending on covariant functors to be treated:
\begin{itemize}
\item \cite [\S 8.3]{FM} requiring \emph{one more extra condition}, 
\item \cite[\S 17.1]{Ful} and \cite{AP, AGP} requiring \emph{two more extra conditions} and 
\item \cite{GK} requiring \emph{three more extra conditions}.
\end{itemize}
In this paper we do not consider these refined ones for the sake of simplicity.

In this paper, motivated by the definition of an operational bivariant theory given in \cite[\S 8.1]{FM}, we define \emph{a co-operational bivariant theory}\footnote{A finer co-operational bivariant theory, motivated by the other refined or sophisticated operational bivariant theories listed above will be given and discussed in a different paper.} from \emph{a contravariant functor} which is not necessarily a multiplicative cohomology theory. Let $F^*$ be a contravariant functor such as the usual cohomology theory $H^*$, the K-theory $K(X)$ of complex vector bundles. Then the co-operational bivariant theory $\bB^{coop}F^*(X \xrightarrow f Y)$ is defined in a way analogous to the definition of Fulton--MacPherson's operational bivariant theory associated to a covariant functor $F_*$. In this paper Fulton--MacPherson's operational bivariant theory shall be denoted by $\bB^{op}F^*(X \xrightarrow f Y)$. For example, let $F^*$ be the usual cohomology groups $H^*$. It turns out that for the identity $\op{id}_X:X \to X$ the co-operational bivariant 
theory $\bB^{coop}H^*(X \xrightarrow {\op{id}_X} X)$ consists of elements $c$'s such that each $c$ is a collection of homomorphisms $c_g: H^*(X') \to H^*(X')$ for all $g:X' \to X$ such that for maps 
\begin{equation}\label{maps}
X'' \xrightarrow h X' \xrightarrow g X
\end{equation}
the following diagram commutes:
\begin{equation}\label{cd1}
\xymatrix
{H^*(X'') \ar[d]_{c_{g \circ h}} && H^*(X')\ar[ll]_{h^*} \ar[d]^{c_g}\\
H^*(X'') && H^*(X') \ar[ll]^{h^*}.
}
\end{equation}
In other words, $c_g: H^*(X') \to H^*(X')$ is nothing but what is usually called \emph{a cohomology operation}\footnote{Speaking of ``operation", in \cite{Vishik2} A. Vishik considers an operation between two oriented cohomology theories \cite{LM} (cf. \cite{LP}): An operation from $A^*$ to $B^*$ is a natural transformation from $A^*$ to $B^*$ considered as \emph{contravariant functors} on the category of smooth schemes, i.e., operations commute with \emph{pull-backs (but not necessarily with push-forwards).} On the other hand, in \cite{Sechin} P. Sechin considers \emph{a multiplicative operation} between two oriented cohomology theories.} (e.g., see \cite{Steenrod1, Steenrod2, Steenrod3} and \cite{nLab2}) for the given cohomology group $H^*$. Here we note that we consider such operations for all objects $X'$ over $X$, i.e., $X' \to X$ (in other words, we consider such operations on what is usually called \emph{the category over $X$}), as we define above. In this paper, we can consider a general contravariant functor $F^*$ with values in a certain category. As Fulton and MacPherson consider a homology theory, i.e., a covariant functor with values in the category of graded abelian groups in \cite[\S 8.1]{FM}, in \S \ref{coop} below we consider a contravariant functor $F^*$ with values in the category of graded abelian groups as a model case. So, in this case, $c_g: F^*(X') \to F^*(X')$ is a homomorphism of graded abelian groups. If $F^*$ is a contravariant functor with values in an arbitrary category, such as the category of sets, abelian groups, derived categories, then  $c_g: F^*(X') \to F^*(X')$ is \emph{a morphism} in the category. Even if we consider such an arbitrary category, a morphism $c_g: F^*(X') \to F^*(X')$  shall be still called \emph{``cohomology" operation}, although it should be called ``a contravariant functor operation" if we ``follow" the way of naming ``cohomology operation". 

A natural transformation $\tau: F^* \to G^*$ between two contravariant functors can be extended to a Grothendieck transformation
$$\tau^{coop}: \bB^{coop}_{\tau}F^*(X \xrightarrow f Y) \to \bB^{coop}G^*(X \xrightarrow f Y)$$
where $\bB^{coop}_{\tau}F^*(X \xrightarrow f Y)$ is a subgroup of $\bB^{coop}F^*(X \xrightarrow f Y)$
and \emph{depends on $\tau$}. 
In particular, for the identity map $\op{id}_X:X \to X$, 
$$\bB^{coop}_{\tau}F^*(X \xrightarrow {\op{id}_X} X)$$
consists of elements $c$'s such that each $c$ is a collections of homomorphisms $c_g:F^*(X') \to F^*(X')$, for all $g:X' \to X$ such that the above diagram (\ref{cd1}) commutes with $H^*$ being replaced by $F^*$ and the following diagram also commutes for some $d \in \bB^{coop}G^*(X \xrightarrow {\op{id}_X} X)$:
\begin{equation}\label{1-cd-FG}
\xymatrix
{ F^*(X') \ar[d]_{c_g } \ar[rr]^{\tau} && G^*(X') \ar[d]^{d_g}  \\
F^*(X') \ar[rr]_{\tau} && G^*(X') \\
}
\end{equation}
And the homomorphism (for the identity $\op{id}_X:X \to X$)
\begin{equation}\label{tau-c-d}
\tau^{coop}: \bB^{coop}_{\tau}F^*(X \xrightarrow {\op{id}_X} X) \to \bB^{coop}G^*(X \xrightarrow {\op{id}_X} X) \quad \text{defined by} \quad \tau^{coop}(c):=d
\end{equation}
also means the above commutative diagrams (\ref{1-cd-FG}) for each $g:X' \to X$. 

A natural transformation $\tau: F^* \to G^*$ of contravariant functors $F^*$ and $G^*$, of course, means that it assigns a homomorphism $\tau: F^*(X) \to F^*(X)$ to each object $X$ and the following diagram commutes for $h:X'' \to X'$:
\begin{equation}\label{cd-FG-naturaliy}
\xymatrix
{ F^*(X') \ar[d]_{h^*} \ar[rr]^{\tau} && G^*(X') \ar[d]^{h^*}  \\
F^*(X'') \ar[rr]_{\tau} && G^*(X'') \\
}
\end{equation}

Obviously the homomorphism $\tau: F^*(X) \to G^*(X)$ for each $X$ can be interpreted as the following trivial commutative diagram 
\begin{equation}\label{cd-FG-id}
\xymatrix
{ F^*(X) \ar[d]_{\op{id}_{F^*(X)}} \ar[rr]^{\tau} && G^*(X) \ar[d]^{\op{id}_{G^*(X)}}  \\
F^*(X) \ar[rr]_{\tau} && G^*(X). 
}
\end{equation}
If we adopt this interpretation, then the above commutative diagram (\ref{cd-FG-naturaliy}) meaning \emph{``naturality" of $\tau$ of two contravariant functors} can be expressed as the following commutative cubic diagram:
\begin{equation}\label{1-cd-FGX-X'-X''}
\xymatrix
{F^*(X') \ar[dd]_{\op{id}_{F^*(X')}} \ar[rd]^{h^*} \ar[rr]^{\tau} && G^*(X') \ar'[d][dd]^(.4){\op{id}_{G^*(X')}} \ar[rd]^{h^*} \\
& F^*(X'') \ar[dd]_(.3){ \op{id}_{F^*(X'')}} \ar[rr]^(.4){\tau}  &&  G^*(X'') \ar[dd]^{\op{id}_{G^*(X'')}} \\
F^*(X') \ar'[r] [rr]_{\tau \quad \quad }  \ar[rd]_{h^*} && G^*(X') \ar[rd]_{h^*} \\
& F^*(X'') \ar[rr] _{\tau }  &&  G^*(X'').
}
\end{equation}
Here we emphasize that the identity map $\op{id}_{F^*(X)}:F^*(X) \to F^*(X)$ for any contravariant functor $F^*$ is obviously a cohomology operation, which shall be called the ``\emph{identity cohomology operation}" in this paper.
Let us consider the following class $\jeden_X^{F^*} \in \bB^{coop}_{\tau}F^*(X \xrightarrow {\op{id}_X} X)$ of 
identity cohomology operations:
$$\jeden_X^{F^*}:= \{ \op{id}_{F^*(X')}:F^*(X') \to F^*(X') \, \, | \, \, g:X' \to X \}$$
Hence, $\op{id}_{F^*(X')} = (\jeden_X^{F^*})_g$ if we use the above notation $c_g$, and  
the above diagram (\ref{1-cd-FGX-X'-X''}) becomes
\begin{equation}\label{1-cd-FGX-X'-X''-b}
\xymatrix
{F^*(X') \ar[dd]_{(\jeden_X^{F^*})_g} \ar[rd]^{h^*} \ar[rr]^{\tau} && G^*(X') \ar'[d][dd]^(.4){(\jeden_X^{G^*})_g} \ar[rd]^{h^*} \\
& F^*(X'') \ar[dd]_(.3){(\jeden_X^{F^*})_{g \circ h}} \ar[rr]^(.4){\tau}  &&  G^*(X'') \ar[dd]^{(\jeden_X^{G^*})_{g \circ h}} \\
F^*(X') \ar'[r] [rr]_{\tau \quad \quad }  \ar[rd]_{h^*} && G^*(X') \ar[rd]_{h^*} \\
& F^*(X'') \ar[rr] _{\tau }  &&  G^*(X'').
}
\end{equation}
Therefore, by (\ref{tau-c-d}) we get
\begin{equation}\label{tau-F-G}
\tau^{coop}(\jeden_X^{F^*})= \jeden_X^{G^*}.
\end{equation}
Namely, the natural transformation $\tau:F^* \to G^*$, in other words (\ref{1-cd-FGX-X'-X''}), can be interpreted as (\ref{tau-F-G}), using the Grothendieck transformation $\tau^{coop}: \bB^{coop}_{\tau}F^* \to \bB^{coop}G^*$ of the associated co-operational bivariant theories. 

Here we give two more non-trivial examples:
Firstly, consider the Chern character $ch: K(-) \to H^{ev}(-;\bQ)$ (e.g., see \cite[\S 3 Properties]{nLab}). Then for each positive integer $k$ we have the Adams operation
$$\Psi^k: K(-) \to K(-)$$
and the Adams-like operation on the even-degree rational cohomology
$$\Psi^k_H:H^{ev}(-;\bQ) \to H^{ev}(-;\bQ)$$
which is defined by $\Psi^k_H(a):=k^r \cdot a$ for each $a \in H^{2r}(-;\bQ)$
and we have the following commutative diagram
\begin{equation}\label{cd3}
\xymatrix
{K(X) \ar[d]_{\Psi^k} \ar[rr]^{ch} && H^{ev}(X;\bQ) \ar[d]^{\Psi^k_H}\\
K(X) \ar[rr]_{ch} && H^{ev}(X;\bQ)
}
\end{equation}
which is natural with respect to the base change, namely we do have the following commutative cube for $h:X'' \to X'$:
\begin{equation}\label{ccube2}
\xymatrix
{K(X') \ar[dd]_{\Psi^k} \ar[rd]^{h^*} \ar[rr]^{ch} && H^{ev}(X';\bQ) \ar'[d][dd]^(.3){\Psi^k_H} \ar[rd]^{h^*} \\
& K(X'') \ar[dd]_(.3){\Psi^k} \ar[rr]^(.3){ch}  &&  H^{ev}(X'';\bQ) \ar[dd]^{\Psi^k_H} \\
K(X') \ar'[r] [rr]_{ch \quad \quad }  \ar[rd]_{h^*} && H^{ev}(X';\bQ) \ar[rd]_{h^*} \\
& K(X'') \ar[rr] _{ch \quad \quad  \quad }  &&  H^{ev}(X'';\bQ).
}
\end{equation}
The above commutative diagram (\ref{cd3}) is nothing but \emph{a part} of the homomorphism (see Remark \ref{rem} below)
$$ch^{coop}: \bB^{coop}_{ch}K^*(X \xrightarrow {\op{id}_X} X) \to \bB^{coop}H^*(X \xrightarrow {\op{id}_X} X).$$
\begin{rem}\label{rem}
Here we remark that the Adams operation $\Psi^k: K(-) \to K(-)$ is a ring-homomorphism, hence the addition of two Adams operations is not necessarily an Adams operation, since the addition of two ring-homomorphisms is not necessarily a ring-homomorphism because in general $(f+g)(xy)\not = (f+g)(x)(f+g)(y)$, although the composition of two Adams operations is an Adams operation. Thus $\bB^{coop}K^*(X \xrightarrow {\op{id}_X} X)$ is \emph{a set and cannot be an Abelian group}, provided that the contravariant functor $K(-)$ is considered to have values in the category of commutative rings. However, if we consider only the abelian group structure, ignoring the structure of product making $K(-)$ a ring, then $\bB^{coop}K^*(X \xrightarrow {\op{id}_X} X)$ \emph{is an Abelian group}.
\end{rem}
Here we emphasize that ``a part" means that it is possible that there are many other cohomology operations on $K$-theory and rational even-dimensional cohomology and natural transformations of them like (\ref{cd3}). 
For example, in (\ref{cd3}) $\Psi^k$ and $\Psi^k_H$ can be respectively replaced by the identity cohomology operations 
$\op{id}_{K(X)}$ and $\op{id}_{H^{ev}(X, \mathbb Q)}$.
Therefore 
$$ch^{coop}: \bB^{coop}_{ch}K^*(X \to Y) \to \bB^{coop}H^*(X \to Y).$$
is an extension of the above natural transformation (\ref{cd3}) of cohomology operations of $K$-theory and rational even-dimensional cohomology to a Grothendieck transformation.

Secondly, we consider the projection $\overline{pr}: \Omega^*(-) \to CH^*(-)/2$ from Levine--Morel's algebraic cobordism $\Omega^*(X)$ to the Chow group $CH^*(X)/2$, which is the Chow cohomology group $CH^*(X)/2$ modded out by 2-torsions (see \cite{Vishik}). Then, as shown in \cite{Vishik} (see Brosnan \cite{Bros}, M. Levine \cite{Levine} and A. Merkurjev \cite{Merk})  there exists (unique) operations $S^i:CH^*(X)/2 \to CH^{*+i}(X)/2$, called ``Steenrod operation", such that the following diagram commutes:
\begin{equation}\label{cd-omega}
\xymatrix
{\Omega^*(X) \ar[d]_{S^i_{LN}} \ar[rr]^{\overline{pr}} && CH^*(X)/2 \ar[d]^{S^i}\\
\Omega^*(X) \ar[rr]_{\overline{pr}} && CH^*(X)/2
}
\end{equation}
Here $S^i_{LN}$ is a Landwever--Novikov operation (see \cite{LM}) and both Steenrod operation and Landwever--Novikov operations commute with pullbacks, namely we have a commutative cube as the above (\ref{ccube2}), which is, as in the case of the above Chern  character, nothing but a part of the homomorphism
$$pr^{coop}: \bB^{coop}_{\overline{pr}}\Omega^*(X \xrightarrow {\op{id}_X} X) \to \bB^{coop}CH^*(X \xrightarrow {\op{id}_X} X)/2.$$
Hence we have the Grothendieck transformation
$$pr^{coop}: \bB^{coop}_{\overline{pr}}\Omega^*(X \to Y) \to \bB^{coop}CH^*(X \to Y)/2$$
is an extension of the above natural transformation (\ref{cd-omega}) of Levine--Morel's algebraic cobordism  and the Chow group $CH^*(X)/2$ to a Grothendieck transformation.

The up-shot is that 
\begin{enumerate}
\item \emph{a co-operational bivariant theory ``captures" or ``interprets" cohomology operations ``bivariant theoretically"}, or \emph{a co-operational bivariant theory is a reasonable general setting to study cohomology operations} and 
\item a Grothendieck transformation
$$\tau^{coop}: \bB^{coop}_{\tau}F^*(X \xrightarrow f Y) \to \bB^{coop}G^*(X \xrightarrow f Y)$$
is an extension or a generalization of a natural transformation of two kind cohomology operations, as shown by the above two examples.
\end{enumerate}

Finally we remark that Fulton--MacPherson's operational bivariant theory is actually described by using the notion of ``homology operation\footnote{In \cite[3.1 Definition]{HM} R. Hardt and C. McCrory define ``stable homology operation" as a natural transformation of homology functor compatible with the suspension isomorphism, just like a stable cohomology operation, which is a natural transformation of cohomology functor compatible with the suspension isomorphism. Hence it is quite natural that ``homology operation" is considered as a natural transformation of homology functor, just like cohomology operation. In fact, in \cite[Example 4.1.25]{LM} the above Landwever--Novikov operation is described as a homology operation.}" instead of ``cohomology operation". In this sense both theories are really \emph{``operational"} bivariant theories. Furthermore works will be done in a different paper.

\section{Fulton--MacPherson's bivariant theories}\label{BT}

We make a quick review of Fulton--MacPherson's bivariant theory \cite {FM}, since we refer to some axioms required on the theory in later sections.

Let $\mathscr C$ be a category which has a final object $pt$ and on which the fiber product or fiber square is well-defined. Also we consider the following classes:
\begin{enumerate}
\item a class $\mathcal C$ of maps, called ``confined maps" (e.g., proper maps, in algebraic geometry), which are \emph{closed under composition and base change, and contain all the identity maps}, and 
\item a class $\mathcal Ind$ of commutative diagrams, called ``independent squares" (e.g., fiber square, ``Tor-independent" square, in algebraic geometry), 
satisfying that

(i) if the two inside squares in  
$$\CD
X''@> {h'} >> X' @> {g'} >> X \\
@VV {f''}V @VV {f'}V @VV {f}V\\
Y''@>> {h} > Y' @>> {g} > Y \endCD
\quad \quad \qquad \text{or} \qquad \quad \quad 
\CD
X' @>> {h''} > X \\
@V {f'}VV @VV {f}V\\
Y' @>> {h'} > Y \\
@V {g'}VV @VV {g}V \\
Z'  @>> {h} > Z \endCD
$$
are independent, then the outside square is also independent,

(ii) any square of the following forms are independent:
$$
\xymatrix{X \ar[d]_{f} \ar[r]^{\op {id}_X}&  X \ar[d]^f & & X \ar[d]_{\op {id}_X} \ar[r]^f & Y \ar[d]^{\op {id}_Y} \\
Y \ar[r]_{\op {id}_X}  & Y && X \ar[r]_f & Y}
$$
where $f:X \to Y$ is \emph{any} morphism. 
\end{enumerate}

\begin{rem}Given an independent square, 
its transpose is \emph{not necessarily} independent. For example, let us consider the category of topological spaces and continuous maps. Let \emph{any} map be confined, and we allow a fiber square
\quad $\CD
X' @> {g'} >> X \\
@V {f'}VV @VV {f}V\\
Y' @>> {g} > Y \endCD
$ \quad 
to be \emph{independent only if $g$ is proper} (hence $g'$ is also proper). Then its transpose is \emph{not independent unless 
$f$ is proper}. (Note that the pullback of a proper map by any continuous map is proper, because ``proper" is equivalent to ``universally closed", i.e., the pullback by any map is closed.)
\end{rem}
\begin{def}\label{BivariantTheory}
A \emph{bivariant theory} $\mathbb B$ on a category $\mathscr C$ with values in the category of graded abelian groups\footnote{As we mentioned in Introduction, instead of graded abelian groups, we consider also sets, e.g., such as the set of complex structures and the set of Spin structures (see \cite[\S 4.3.2]{FM}),  and categories, e.g., such as the derived (triangulated) category of $f$-perfect complexes (see \cite[\S 7.1 Grothendieck duality]{FM}) as well.}
 is an assignment to each morphism
$ X  \xrightarrow{f} Y$
in the category $\mathscr C$ a graded\footnote{The grading is sometimes ignored. In this case we can consider that the grading is only $0$, i.e, $\bB(X \to Y) = \bB^0(X \to Y)$. For an example for such a case, see \cite[\S 6.1.2 Definition of $\mathbb F$]{FM} where $\mathbb F(X \to Y) = \mathbb F^0(X \to Y)$.}
abelian group $$\bB(X  \xrightarrow{f} Y)$$
which is equipped with the following three basic operations. The $i$-th component of $\bB(X  \xrightarrow{f} Y)$, $i \in \mathbb Z$, is denoted by $\bB^i(X  \xrightarrow{f} Y)$.
\begin{enumerate}
\item {\bf Product}: For morphisms $f: X \to Y$ and $g: Y
\to Z$, the product 
$$\bullet: \bB^i( X  \xrightarrow{f}  Y) \otimes \bB^j( Y  \xrightarrow{g}  Z) \to
\bB^{i+j}( X  \xrightarrow{g\circ f}  Z).$$
\item {\bf Pushforward}: For morphisms $f: X \to Y$
and $g: Y \to Z$ with $f$ \emph {confined}, 

the pushforward 
$f_*: \bB^i( X  \xrightarrow{g\circ f} Z) \to \bB^i( Y  \xrightarrow{g}  Z).$
\item {\bf Pullback} : For an \emph{independent} square \qquad $\CD
X' @> g' >> X \\
@V f' VV @VV f V\\
Y' @>> g > Y, \endCD
$

the pullback 
$g^* : \bB^i( X  \xrightarrow{f} Y) \to \bB^i( X'  \xrightarrow{f'} Y'). $
\end{enumerate}

An element $\alp \in \bB(X \xrightarrow f Y)$ is sometimes expressed as follows:
\[
\xymatrix
{
X \ar[rr]_f^{\maru{$\alp$}} && Y
} 
\]

These three 
operations are required to satisfy the following seven compatibility 
axioms (\cite [Part I, \S 2.2]{FM}):

\begin{enumerate}
\item[($A_1$)] {\bf Product is associative}: for 
\xymatrix
{
X \ar[r]_f^{\maru{$\alp$}} & Y \ar[r]_g^{\maru{$\be$}} & Z \ar[r]_h^{\maru{$\ga$}} & Z
}
$$(\alp \bullet\be) \bullet \ga = \alp \bullet (\be \bullet \ga).$$
\item[($A_2$)] {\bf Pushforward is functorial}:for 
\xymatrix
{
X  \ar@/^10pt/[rrr]^{\maru{$\alp$}}  \ar[r]_f & Y \ar[r]_g  & Z \ar[r]_h  & W
}
with confined $f, g$, 
$$(g\circ f)_* \alp = g_*(f_*\alp).$$
\item[($A_3$)] {\bf Pullback is functorial}: given independent squares
$$\CD
X''@> {h'} >> X' @> {g'} >> X \\
@VV {f''}V @VV {f'}V @VV {f}V\\
Y''@>> {h} > Y' @>> {g} > Y \endCD
$$
$$(g \circ h)^* = h^* \circ g^*.$$
\item[($A_{12}$)] {\bf Product and pushforward commute}: $f_*(\alp \bullet\be)  = f_*\alp \bullet \be$ \\

for 
\xymatrix
{
X \ar@/^10pt/[rr]^{\maru{$\alp$}}  \ar[r]_f & Y \ar[r]_g  & Z \ar[r]_h^{\maru{$\be$}} & W
}
with confined $f$, 
\item[($A_{13}$)] {\bf Product and pullback commute}: \quad $h^*(\alp \bullet\be)  = {h'}^*\alp \bullet h^*\be$ \quad for independent squares
 \[
\xymatrix{X' \ar[rrr]^{h''} \ar[d]_{f'}
 &&& X \ar[d]_f^{\maru{$\alp$}}   \\
Y' \ar[d]_{g'}\ar[rrr]^{h'} &&& Y \ar[d]_g^{\maru{$\be$}} \\
Z' \ar[rrr]_h &&& Z
} 
\]
\item[($A_{23}$)] 
{\bf Pushforward and pullback commute}: \quad $f'_*(h^*\alp)  = h^*(f_*\alp)$ \quad for independent squares with $f$ confined 
\[
\xymatrix{
X' \ar[rrr]^{h''} \ar[d]_{f'}  &&& X \ar[d]_f \ar@/^11pt/[dd]^{\maru{$\alp$}} \\
Y' \ar[d]_{g'} \ar[rrr]^{h'} &&& Y \ar[d]_g\\
Z' \ar[rrr]_h &&& Z
} 
\]
\item[($A_{123}$)] {\bf Projection formula}: \, $g'_*(g^*\alp \bullet \be)  = \alp \bullet g_*\be$ \quad for an independent square with $g$ confined 
\[
\xymatrix{
X' \ar[r]^{g'} \ar[d]^{f'}_{\Maru{$g^*\alp$}} 
& X \ar[d]_f^{\maru{$\alp$}}  \\
Y' \ar@/_18pt/[rrrrr]^{\maru{$\beta$}} \ar[r]^g & Y \ar[rrrr]_h^{\quad \Maru{$g_*\be$}} &&&& Z
} 
\]
\end{enumerate}
We also require the theory $\bB$ to have multiplicative units:
\begin{enumerate}
\item[({\bf Units})] For all $X \in \mathscr C$, there is an element $1_X \in \bB^0( X  \xrightarrow{\op {id}_X} X)$ such that $\alp \bullet 1_X = \alp$ for all morphisms $W \to X$ and all $\alp \in \bB(W \to X)$, and such that $1_X \bullet \beta = \beta $ for all morphisms $X \to Y$ and all $\beta \in \bB(X \to Y)$, and such that $g^*1_X = 1_{X'}$ for all $g: X' \to X$.
\end{enumerate}
\end{def}

A bivariant theory unifies both a covariant theory and a contravariant theory in the following sense:
For a bivariant theory $\bB$, its associated covariant functors and contravariant functors are defined as follows:
\begin{enumerate}
\item $\bB_*(X):= \bB(X \xrightarrow {a_X} pt)$ is covariant for confined morphisms and the grading is given by $\bB_i(X):= \bB^{-i}(X \xrightarrow {a_X} pt)$.
\item $\bB^*(X) := \bB(X  \xrightarrow{\op{id}_X}  X)$ is contravariant for all morphisms and the grading is given by $\bB^j(X):= \bB^j(X   \xrightarrow{\op{id}_X}  X)$.
\end{enumerate}

A typical example of a bivariant theory is the bivariant homology theory $\bH(X \xrightarrow f Y)$ 
constructed from the singular cohomology theory $H^*(-)$, which unifies the Borel--Moore homology $H_*^{BM}(X):=\bH^{-*}(X \to pt)$ and the singular cohomology $H^*(X):=\bH^*(X \xrightarrow {\op{id}_X} X)$. Here the underlying category $\mathscr C$ is the category of spaces embeddable as closed subspaces of some Euclidean spaces $\mathbb R^n$ and continuous maps between them (see \cite[\S 3 Topological Theories]{FM}).
More generally, Fulton--MacPherson's (general) bivariant homology theory 
$$h^*(X \to Y)$$
(here, using their notation) 
is constructed from \emph{a multiplicative cohomology theory $h^*(-)$} \cite[\S 3.1]{FM}
Here the cohomology theory $h^*$ is either ordinary or generalized. 
A cohomology theory $h^*$ is called \emph{multiplicative} if for pairs $(X,A), (Y,B)$ there is a graded pairing (exterior product)
$$h^i(X,A) \times h^j(Y,B) \xrightarrow {\times} h^{i+j}(X \times Y, X \times B \sqcup A \times Y)$$
such that it is associative and graded commutative, i.e., $\alp \times \be = (-1)^{i+j} \be \times \alp$.
A typical example of a multiplicative ordinary cohomology theory is the singular cohomology theory. The topological complex $K$-theory $K(-)$ and complex cobordism theory $\Omega^*(-)$ are  multiplicative generalized  cohomology theories. We consider the category of spaces embeddable as closed subspaces in some Euclidian spaces $\mathbb R^N$ and continuous maps. For example, \emph{Whitney's embedding theorem} says that any manifold of real dimension $m$ can be embedded as a closed subspace of $\mathbb R^{2m}$. We also note that a complex algebraic variety is embeddable\footnote{This is because the variety $X$ is covered by finitely many affine varieties, which are embedded (as closed subsets) into $\mathbb R^n$ for some $n$, thus it follows from \cite[\S 8.8 Proposition]{Dold} that the variety X is itself embedded (as a closed subset) into $\mathbb R^N$ for some $N$. }  as a closed subspace of some Euclidean space $\mathbb R^N$. We let a confined map be a proper map and an independent square be a fiber square.
\begin{def}\label{biv-hom}
For a continuous map $f:X \to Y$, choose a map $\phi: X \to \mathbb R^n$ such that $\Phi=(f, \phi): X \to Y \times \mathbb R^n$ defined by $\Phi(x):=(f(x), \phi(x))$ is  a closed embedding\footnote{ Such a map always exist, since our space is embeddable as a closed subspace of some $\mathbb R^N$, thus this embedding is considered as $\phi:X \to \mathbb R^N$, then $\Phi=(f, \phi)$ is also a closed embedding.}. Then we define
\begin{equation}\label{embed}
h^*(X \to Y):= h^{i+n}(Y \times \mathbb R^n, Y \times \mathbb R^n \setminus \Phi(X)).
\end{equation}
\end{def}
\begin{thm}(\cite[p.34-p.38]{FM}) The above definition (\ref{embed}) is independent of the choice of the embedding $\phi$, thus $\Phi$, and $h^*(X \to Y)$ is a bivariant theory.
\end{thm}
\begin{rem} 
\begin{enumerate}
\item By the definition (\ref{embed}) we have
$h^i(X \xrightarrow {\op{id}_X} X) = h^i(X).$
Indeed, since $\op{id}_X: X \to X$ is obviously a closed map (embedding), we can choose $\phi:X \to pt$ so that
$\Phi=(\op{id}_X, \phi):X \to X \times pt \cong X$ is a closed map. Hence 
we have 
$$h^i(X \xrightarrow {\op{id}_X} X)= h^i(X, (X \times pt) \setminus \Phi(X))=h^i(X, \emptyset)=h^i(X).$$
\item $h^{-i}(X \xrightarrow {a_X} pt) = h^{n-i}(\mathbb R^n, \mathbb R^n \setminus \Phi(X))=:h^{n-i}(\mathbb R^n, \mathbb R^n \setminus X)$ where $\Phi=(a_X, \phi):X \to pt \times \mathbb R^n =\mathbb R^n$ is a closed embedding. If $h^*=H^*$ is the singular cohomology, then $h^{-i}(X \xrightarrow {a_X} pt) =h^{n-i}(\mathbb R^n, \mathbb R^n \setminus X) =H^{n-i}(\mathbb R^n, \mathbb R^n \setminus X) =:H^{BM}_i(X)$ is the Borel--More homology group (e.g., see \cite{Ful}, \cite[B.1]{PS}).
\end{enumerate}
\end{rem}
\begin{defn}
(\cite[\S 2.7 Grothendieck transformation]{FM}) \label{groth}
Let $\bB, \bB'$ be two bivariant theories on a category $\mathscr C$. 
A {\it Grothendieck transformation} from $\bB$ to $\bB'$, $\ga : \bB \to \bB'$
is a collection of homomorphisms
$\bB(X \to Y) \to \bB'(X \to Y)$
for a morphism $X \to Y$ in the category $\mathscr C$, which preserves the above three basic operations: 
\begin{enumerate}
\item $\ga (\alp \bullet_{\bB} \be) = \ga (\alp) \bullet _{\bB'} \ga (\be)$, 
\item $\ga(f_{*}\alp) = f_*\ga (\alp)$, and 
\item $\ga (g^* \alp) = g^* \ga (\alp)$. 
\end{enumerate}
\end{defn}

A Grothendieck transformation $\ga: \bB \to \bB'$ induces natural transformations $\ga_*: \bB_* \to \bB_*'$ and $\ga^*: \bB^* \to {\bB'}^*$.
 
\begin{rem}(see \cite[\S 3.2 Grothendieck transformations of topological theories]{FM}) \label{contra}
Let $t: h^* \to \widetilde h^*$ be a natural transformation of two multiplicative cohomology theories. Then we get the associated Grothendieck transformation
\begin{equation}\label{bivari-h}
t:h^*( X \xrightarrow f Y) \to \widetilde h^*( X \xrightarrow f Y) 
\end{equation}
since we have $t: h^{*+n}(Y \times \mathbb R^n, Y \times \mathbb R^n \setminus \Phi(X)) \to \widetilde h^*(Y \times \mathbb R^n, Y \times \mathbb R^n \setminus \Phi(X))$. For example, the Chern character
$ch: K^0(-) \to H^*(-) \otimes \mathbb Q$ induces the Grothendieck transformation
$ch: K^0(X \xrightarrow f Y) \to H^*(X \xrightarrow f Y) \otimes \mathbb Q$ (see \cite[\S 3.2.2 Examples]{FM}).
\end{rem}
For later use, we introduce the following ``image'' of a Grothendieck transformation $\ga: \mathbb B \to \mathbb B'$:
$$\op{Im} \ga := \op{Image}(\ga: \mathbb B \to \mathbb B'),$$
which is defined by, for a map $f:X \to Y$,
$$(\op{Im}\ga) (X \xrightarrow f Y):=  \op{Image}\left (\ga: \mathbb B(X \xrightarrow f Y)  \to \mathbb B'(X \xrightarrow f Y) \right ) \subset \mathbb B'(X \xrightarrow f Y).$$
Then it is easy to see that $\op{Im}\ga$ is a bivariant subtheory of $\mathbb B'$. This in fact follows from the above three properties: (1) $\ga (\alp \bullet_{\bB} \be) = \ga (\alp) \bullet _{\bB'} \ga (\be)$, (2) $\ga(f_{*}\alp) = f_*\ga (\alp)$ and (3) $\ga (g^* \alp) = g^* \ga (\alp)$. Indeed, we have: 
\begin{enumerate}
\item The bivariant product $(\op{Im}\ga) (X \xrightarrow f Y) \otimes (\op{Im}\ga) (Y \xrightarrow g Z) \xrightarrow {\bullet} (\op{Im}\ga) (Y \xrightarrow {g \circ f}  Z)$ is well-defined, since the following diagram commutes because of $\ga (\alp \bullet_{\bB} \be) = \ga (\alp) \bullet _{\bB'} \ga (\be)$:
$$
\xymatrix
{
\mathbb B(X \xrightarrow f Y) \otimes \mathbb B(Y \xrightarrow g Z)   \ar[rr]^{\qquad \bullet_{\mathbb B}} \ar[d]_{\ga \otimes \ga} && \mathbb B(Y \xrightarrow {g \circ f}  Z) \ar[d]^{\ga}\\
(\op{Im}\ga) (X \xrightarrow f Y) \otimes (\op{Im}\ga) (Y \xrightarrow g Z)   \ar[rr]_{\qquad \bullet_{\mathbb B'}} && (\op{Im}\ga) (Y \xrightarrow {g \circ f}  Z).
}
$$
\item The pushforward $f_*: (\op{Im}\ga) (X \xrightarrow {g \circ f} Z) \to (\op{Im}\ga) (Y \xrightarrow g Z)$  is well-defined, since the following diagram commutes because of $f_*(\ga(\alp))=\ga(f_*\alp)$:
$$
\xymatrix
{
\mathbb B(X \xrightarrow {g \circ f} Z) \ar[rr]^{f_*} \ar[d]_{\ga} && \mathbb B(Y \xrightarrow g Z) \ar[d]^{\ga}\\
(\op{Im}\ga) (X \xrightarrow {g \circ f} Z) \ar[rr]_{f_*} && (\op{Im}\ga) (Y \xrightarrow g Z).
}
$$
\item The pullback $g^*: (\op{Im}\ga) (X \xrightarrow f  Y) \to (\op{Im}\ga) (X' \xrightarrow {f'} Y')$  is well-defined, since the following diagram commutes because of $g^*(\ga(\alp))=\ga(g^* \alp)$:
$$
\xymatrix
{
\mathbb B(X \xrightarrow f Y) \ar[rr]^{g^*} \ar[d]_{\ga} && \mathbb B(X' \xrightarrow {f'} Y') \ar[d]^{\ga}\\
(\op{Im}\ga) (X \xrightarrow f Y)) \ar[rr]_{g^*} && (\op{Im}\ga) (X' \xrightarrow {f'} Y').
}
$$
Here we consider the fiber square \quad 
$\CD
X' @> {g'} >> X \\
@V {f'}VV @VV {f}V\\
Y' @>> {g} > Y. \endCD
$
\end{enumerate}
Unless some confusion is possible, we may use the symbol $\op{Im}_{\ga} \mathbb B'$ for $\op{Im}\ga$ \emph{in order to emphasize that $\op{Im}\ga$ is a subtheory of $\mathbb B'$, thus recording $\mathbb B'$}.

\section{Operational bivariant theory}\label{obt}
Given a covariant functor or a homology theory, Fulton and MacPherson have defined what is called \emph{an operational bivariant theory} or \emph{an operational theory} \cite[\S 8 Operational Theories]{FM} (also, see \cite[\S 17. 1 and \S 17.2]{Ful}). As we remarked in Introduction, in this paper we consider the operational bivariant theory defined in \cite[\S 8.1]{FM}. The case of the other refined operational bivariant theories will be considered in a different paper in which we define
a co-operational bivariant theory ``corresponding" to these refined operational bivariant theories.

Let $h_*$ be a homology theory, i.e., a covariant functor with values in graded abelian groups such that the functorial (pushforward) homomorphism $f_*:h_*(X) \to h_*(Y)$ is defined for a confined map $f:X \to Y$.  Then its \emph{operational bivariant theory}, denoted by $\bB^{op}h_*(X \xrightarrow f Y)$, is defined as follows. For a map $f:X \to Y$, an element $c \in \bB^{op}h_*^i(X \xrightarrow f Y)$ is defined to be a collection of
homomorphisms\footnote{In \cite[\S 8 Operational Theories]{FM}, $c_g$ is denoted by $c_{Y'}$ and in \cite[\S 17. 1]{Ful} $c_g$ is denoted by $c^{(m)}_g$.}
\begin{equation*}
c_g:h_m(Y') \to h_{m-i}(X')
\end{equation*}
for all $m \in \mathbb Z$, all $g:Y' \to Y$ and the fiber square \, \, \, 
$$\CD
X' @> {g'}  >> X\\
@V {f'} VV @VV fV\\
Y' @>> {g} > Y.\\
\endCD
$$
These homomorphisms $c_g$ are
required to be \emph{compatible with pushforward} (for confined maps), i.e., for a fiber diagram
\, \, \, 
\begin{equation}\label{cd-123}
\CD
X'' @> {k'}  >> X' @> {g'}  >> X\\
@V {f''} VV @VV f'V @VV fV\\
Y'' @>> {k} > Y' @>> {g} > Y.\\
\endCD
\end{equation}
where $k$ is confined and thus $k'$ is confined as well, 
the following diagram commutes
\begin{equation}
\xymatrix
{h_{m-i}(X'')  \ar[rr]^{k'_*} && h_{m-i}(X')\\
h_m(Y'') \ar[u]^{c_{g\circ k}}  \ar[rr]_{k_*} && h_m(Y'). \ar[u]_{c_g}
}
\end{equation}
Here we recall the definitions of the three operational-bivariant-theoretic operations:
\begin{defn}\label{defn-op-biv}
\begin{enumerate}
\item {\bf Product}: The product
$$\bullet: \mathbb B^{op}h_*^i(X \xrightarrow f Y) \otimes \mathbb B^{op}h_*^j(Y \xrightarrow g Z) \to \mathbb B^{op}h_*^{i+j}(X \xrightarrow {g \circ f} Z)$$
is defined by, for $c \in \mathbb B^{op}h_*^i(X \xrightarrow f Y) $ and $d \in \mathbb B^{op}h_*^j(Y \xrightarrow g Z) $,
$$(c \bullet d)_h:= c_{h'} \circ d_h: h_m(Z') \xrightarrow {d_h}  h_{m-j}(Y') \xrightarrow {c_{h'}} h_{m-j-i}(X')=h_{m-(i+j)}(X')$$
Here we consider the following fiber squares:
\begin{equation}\label{cd-123v}
\CD
X' @> {h''} >> X \\
@V {f'}VV @VV {f}V\\
Y' @> {h'} >> Y \\
@V {g'}VV @VV {g}V \\
Z'  @>> {h} > Z. \endCD 
\end{equation}
\item {\bf Pushforward}: For $X \xrightarrow f Y \xrightarrow g Z$ with $f$ being confined 
$$f_*: \mathbb B^{op}h_*^i(X \xrightarrow {g \circ f} Z) \to \mathbb B^{op}h_*^i(Y \xrightarrow g Z)$$
is defined by, for $c \in \mathbb B^{op}h_*^i(X \xrightarrow {g \circ f} Z)$ 
$$(f_*c)_h := (f')_* \circ c_h: h_m(Z') \xrightarrow{c_h} h_{m-i}(X')  \xrightarrow{(f')_*} h_{m-i}(Y').$$
Here we use the above commutative diagram (\ref{cd-123v}).
\item {\bf Pullback}: For a fiber square 
$$\CD
X' @> {g'}  >> X\\
@V {f'} VV @VV fV\\
Y' @>> {g} > Y,\\
\endCD
$$
$$g^*: \mathbb B^{op}h_*^i(X \xrightarrow f Y ) \to \mathbb B^{op}h_*^i(X' \xrightarrow {f'} Y' )$$
is defined by, for $c \in \mathbb B^{op}h_*^i(X \xrightarrow f Y) $
$$(g^*c)_k := c_{g \circ k}: h_m(Y'') \to h_{m-i}(X'').$$
Here we use the above commutative diagrams (\ref{cd-123}).
\end{enumerate}
\end{defn}

Let $\bB$ be a bivariant theory. Then its \emph{associated operational bivariant theory} $\bB^{op}$ is defined to be the operational bivariant theory constructed from the covariant functor $\bB_*(X) = \bB(X \to pt)$. If we use the above notation $\bB^{op}h_*(X \xrightarrow f Y)$, to be more precise, we have
$$\bB^{op}(X \xrightarrow f Y):= \bB^{op}\bB_*(X \xrightarrow f Y).$$
Then we have the following canonical Grothendieck transformation
$$ op:\bB \to \bB^{op}$$
defined by, for each $\alp \in \bB(X \to Y)$,
$$
op(\alp):= \{ (g^*\alp) \bullet : \bB_*(Y') \to \bB_*(X') | g:Y' \to Y \}.
$$
Here we note that $(op(\alp))_g = (g^*\alp)^*:\bB_*(Y') \to \bB_*(X')$, which is defined by 
$(g^*\alp)^*(b):= (g^*\alp) \bullet b$,
for which see \cite[\S 2.5 Gysin homomorphisms]{FM}.

For the sake of convenience for the reader and later presentation, we prove the following:
\begin{pro}\label{map op} The above map $ op:\bB \to \bB^{op}$ is a Grothendieck transformation.
\end{pro}
\begin{proof}
\begin{enumerate}
\item $op(\alp \bullet_{\bB} \beta) =op(\alp) \bullet_{\bB^{op}} op(\beta)$  for $\alp \in \bB(X \xrightarrow f  Y)$ and $\beta \in \bB(Y \xrightarrow g  Z)$. From now we simply write $op(\alp \bullet \beta) =op(\alp) \bullet op(\beta)$. Then for $z' \in \bB_*(Z') = \bB(Z' \to pt)$, considering the above diagram (\ref{cd-123v}), we have 
\begin{align*}
(op(\alp \bullet \beta))_h (z') & = h^*(\alp \bullet \be) \bullet z' \quad \text{(by definition)}\\
& = ((h')^*\alp \bullet h^*\be) \bullet z'  \quad \text{(by Axiom ($A_{13}$))}\\
& = (h')^*\alp \bullet (h^*\be \bullet z' ) \quad \text{(by associativity of product, i.e., Axiom ($A_1$))}\\
& =  (h')^*\alp \bullet ((op(\be))_h(z') )  \quad \text{(by definition)}\\
& =  (op(\alp))_{h'}((op(\be))_h(z')) \quad \text{(by definition)}\\
& =  ((op(\alp))_{h'} \circ (op(\be))_h) (z')  \\
& = (op(\alp) \bullet op(\be))_h (z')  \, \, \text{(by definition of \emph{the operational bivariant product}).}
\end{align*}
Hence we have $(op(\alp \bullet \beta))_h  =  (op(\alp) \bullet op(\be))_h $, i.e., $op(\alp \bullet \beta) =op(\alp) \bullet op(\beta)$.
\item $op(f_*\alp) = f_*op(\alp)$, where we consider the above diagram (\ref{cd-123v}) and $\alp \in \bB(X \xrightarrow {g \circ f} Z)$. For $z' \in \bB_*(Z')=\bB(Z' \to pt)$ we have
\begin{align*}
(op(f_*\alp))_h (z') & = h^*(f_*\alp) \bullet z' \quad \text{(by definition)}\\
& =  ((f')_*h^*\alp) \bullet z'  \quad \text{(by Axiom ($A_{23}$))}\\
& =  (f')_* (h^*\alp \bullet z') \quad \text{(by  Axiom ($A_{12}$))}\\
& =  (f')_* ((op(\alp))_h(z'))  \quad \text{(by definition)}\\
& =  ((f')_* \circ (op(\alp))_h )(z') \\
& = (f_*op(\alp))_h (z')  \,\, \text{(by definition of \emph{the operational bivariant pushforward} )}
\end{align*}.
Hence we have $(op(f_*\alp))_h = (f_*op(\alp))_h$, i.e., $op(f_*\alp)  = f_*op(\alp)$.
\item $op(g^*\alp)=g^*op(\alp)$ for $\alp \in \bB(X \xrightarrow f Y)$. Consider the above commutative diagram (\ref{cd-123}) and let $y'' \in \bB_*(Y'')=\bB(Y'' \to pt)$.
\begin{align*}
(op(g^*\alp))_k(y'') & = (k^*(g^*\alp))\bullet y'' \\
& = (g \circ k)^*\alp \bullet y''  \quad \text{(by the functoriality of pullback, i.e., Axiom ($A_3$))}\\
& = (op(\alp))_{g \circ k} (y'') \quad \text{(by definition)}\\
& = (g^*op(\alp))_k (y'') \, \, \text{(by definition of \emph{the operational bivariant pullback}).}
\end{align*}
Hence we have $(op(g^*\alp))_k = (g^*op(\alp))_k$, i.e., $op(g^*\alp) = g^*op(\alp).$
\end{enumerate}
\end{proof}
\begin{rem}\label{3-rem-bb} In a sense, the operational bivariant product, pushforward and pullback are defined like that in Definition \ref{defn-op-biv} above so that the above map $ op:\bB \to \bB^{op}$ becomes a Grothendieck transformation.
\end{rem}

It is not clear whether one could construct a Grothendieck transformation of the associated operational bivariant theories from a natural transformation of two covariant functors. To be more precise, if $t :h_*(-) \to \widetilde h_*(-)$ is a natural transformation of two covariant functors, then it is not clear whether one could construct a Grothendieck transformation $t:\bB^{op}h_*(X \xrightarrow f Y) \to \bB^{op}\widetilde h_*(X \xrightarrow f Y)$, which is an ``operational bivariant theoretic analogue" of (\ref{bivari-h}). A kind of similar problem is discussed in \cite[\S 8.2]{FM}. Suppose that $\bB$ is a bivariant theory, $h_*$ is a covariant functor and there are homomorphisms $\phi(X): \bB_*(X)=\bB(X \to pt) \to h_*(X)$, covariant for confined maps, and taking $1 \in \bB^*(pt)=\bB_*(pt)$ to $1 \in h_*(pt)$. 
The homomorphism $ev_X:\bB^{op}h_*(X \to pt) \to h_*(X)$ defined by $ev_X(c) := (c_{\op{id}_{pt}})(1)$ where $1 \in h_*(pt)$ is called the evaluation homomorphism. Then a question is whether there exists a unique Grothendieck transformation $\Phi:\bB(X \to Y) \to \bB^{op}h_*(X \to Y)$ such that the associated map $\Phi(X): \bB_*(X) \to \bB^{op}h_*(X \to Y)$ followed by the evaluation map $ev_X:\bB^{op}h_*(X \to pt) \to h_*(X)$, i.e., $ev_X \circ \Phi(X)$, is equal to the given homomorphism $\phi(X):\bB_*(X) \to h_*(X)$. The answer to this question is negative, however a positive answer to a modified question is affirmative, details for which see \cite{Yokura-ope} (cf. \cite{BSY-jpaa}). 

Let $\ga:\bB \to \bB'$ be a Grothendieck transformation between two bivariant theories $\bB$ and $\bB'$. Then we have a natural transformation $\ga_*:\bB_* \to \bB_*'$ between two covariant functors $\bB_*$ and $\bB'_*$. As observed above, it is not clear whether one could construct a Grothendieck transformation $t:\bB^{op}\bB_*(X \xrightarrow f Y) \to \bB^{op}\bB'_*(X \xrightarrow f Y)$, i.e, $t:\bB^{op}(X \xrightarrow f Y) \to \bB'^{op}(X \xrightarrow f Y)$. However, 
if we take smaller subtheories of both $\mathbb B^{op}$ and $\mathbb B'^{op}$, then we can show the following results.
\begin{pro}\label{pro-op} Let $\bB$ be a bivariant theory and $\bB^{op}$ be its associated operational bivariant theory.
\begin{enumerate}
\item Define
$\widehat{\bB}^{op}:= \op{Image}(op:\bB \to \bB^{op}),$
that is 
$$\widehat{\bB}^{op}(X \to Y) := op(\bB(X \to Y))  = \left \{op(\alp) \, \, | \, \, \alp \in \bB(X \to Y) \right \} \subset \bB^{op}(X \to Y).$$
Then $\widehat{\bB}^{op}$ is an operational bivariant theory, which is a subtheory of $\bB^{op}$, and 
$\widehat{op}: \bB \to \widehat{\bB}^{op}$ defined by $\widehat{op}(\alp) = op(\alp)$ and the inclusion 
$\iota^{op}_{\bB}:\widehat{\bB}^{op} \hookrightarrow \bB^{op}$ are Grothendieck transformations.
\item \label{X-to-p-op} For a map to a point, $X \to pt$,  we have the isomorphism:
$$\bB_*(X) = \bB(X \to pt) \cong \widehat{\bB}^{op} (X \to pt).$$
\end{enumerate}
\end{pro}
\begin{proof}
 Here we show only (\ref{X-to-p-op}). Let $\alp, \be \in \bB(X \to pt)$. It suffices to show that $op(\alp) =op(\be)$ implies $\alp = \be$, i.e., $(g^*\alp)^* =(g^*\be)^*$ for all $g:Y \to pt$ implies that $\alp = \be$. Let us consider the case $\op{id}_{pt}:pt \to pt$. Then we are supposed to have $(\op{id}_{pt}^*\alp)^* =(\op{id}_{pt}^*\be)^*$, i.e., $(\op{id}_{pt}^*\alp)^* (x) =(\op{id}_{pt}^*\be)^*(x)$, i.e.,
$$\alp \bullet x =(\op{id}_{pt}^*\alp)^* (x)  = (\op{id}_{pt}^*\be)^*(x) = \be \bullet x$$
 for any $x \in \bB_*(pt)=\bB(pt \to pt)$. In particular, let $x$ be the unit $1_{pt} \in \bB_0(X)=\bB^0(pt \xrightarrow {\op{id}_{pt}} pt)$, then we have
 $$ \alp = \alp \bullet  1_{pt} = \be \bullet 1_{pt} = \be.$$
 Here the left and right equalities are due to the property of the unit (see ({\bf Units}) in \S 2, or \cite[\S 2.2 Axioms for a bivariant theory, p. 22]{FM}.) 
 Therefore we have that $op(\alp) =op(\be)$ implies $\alp = \be$.
 \end{proof} 
Using this subtheory $\widehat{\bB}^{op}$, given a Grothendieck transformation $\ga: \bB \to \bB'$, one might be tempted to (quite naturally) think that there would be a canonical Grothendieck transformation $\widehat{\ga}^{op}: \widehat{\bB}^{op} \to \widehat{\bB'}^{op}$ such that the following diagram commutes:
\[
\xymatrix
{
\alp \in \bB (X \to Y)  \ar[rr]^{\ga} \ar[d]_{\widehat{op}} && \quad \bB' (X \to Y) \ni \ga(\alp) \ar[d]^{\widehat{op}} \\
op(\alp) \in \widehat{\bB}^{op}  (X \to Y) \hspace{0.6cm} \ar[rr]_{\widehat{\ga}^{op}} && \hspace{0.6cm} \widehat{\bB'}^{op} (X \to Y) \ni op(\ga(\alp))
}
\]
where $\widehat{\ga}^{op}$ is 
defined by $\widehat{\ga}^{op} (op(\alp)):=op(\ga(\alp))$. Then it follows from the above Proposition \ref{map op} that 
$\widehat{\ga}^{op}: \widehat{\bB}^{op} \to \widehat{\bB'}^{op}$ is a Grothendieck transformation. However, the definition of $\widehat{\ga}^{op} (op(\alp)):=op(\ga(\alp))$ turns out to be \emph{not well-defined} due to the following reason. 
Let $\alp, \beta \in \bB(X \to Y)$ such that $op(\alp)=op(\be)$. In order for the above definition $\widehat{\ga}^{op} (op(\alp)):=op(\ga(\alp))$ to be well-defined, we should have that $op(\ga(\alp))=op(\ga(\beta))$, but this cannot be automatically guaranteed, although it is the case for $\alp, \be \in  \bB(X \to pt)$ as shown in the proof of Proposition \ref{pro-op} (2). 
Indeed, we get the following commutative diagram:
\[
\xymatrix
{
\bB_*(X') \ar[rr]^{\ga} && \bB'_*(X') \\
\bB_*(Y') \ar[u]^{(g^*\alp)^*} \ar[rr]_{\ga} && \bB'_*(Y'), \ar[u]_{(g^*(\ga(\alp)))^*} 
}
\quad 
\xymatrix
{
\bB_*(X') \ar[rr]^{\ga} && \bB'_*(X') \\
\bB_*(Y') \ar[u]^{(g^*\be)^*} \ar[rr]_{\ga} && \bB'_*(Y'). \ar[u]_{(g^*(\ga(\be)))^*} 
}
\]
which is due to the following computation, e.g., for the first diagram: for $\forall y \in \bB_*(Y')$ 
$$\ga((g^*\alp)^*(y)) = \ga(g^*\alp \bullet y) = \ga(g^*\alp) \bullet \ga(y) = g^*(\ga(\alp))\bullet \ga(y) = (g^*(\ga(\alp)))^*(\ga(y)).$$
Thus, we get $\ga  \circ (g^*\alp)^* = (g^*(\ga(\alp)))^* \circ \ga$. Similarly we get $\ga  \circ (g^*\be)^* = (g^*(\ga(\be)))^* \circ \ga$. Therefore $op(\alp) =op(\be)$ implies that $(g^*(\ga(\alp)))^* \circ \ga = \ga  \circ (g^*\alp)^* = \ga  \circ (g^*\be)^* = (g^*(\ga(\be)))^* \circ \ga$, namely we have
$$(g^*(\ga(\alp)))^* \circ \ga = (g^*(\ga(\be)))^* \circ \ga, $$
which \emph{does not necessarily} imply that $(g^*(\ga(\alp)))^* = (g^*(\ga(\be)))^*$. 
\emph{However, if $\ga: \bB_*(Y') \to \bB'_*(Y')$ is surjective, then we do have $(g^*(\ga(\alp)))^* = (g^*(\ga(\be)))^*$, namely we have
$$op(\ga(\alp)) = op(\ga(\be)).$$
}
Therefore we get the following:
\begin{pro}\label{im-ga}  Let $\ga: \bB \to \bB'$ be a Grothendieck transformation. Then there is a canonical Grothendieck transformation
$$\widehat{\ga}^{op}: \widehat{\bB}^{op} \to \widehat{\op{Im}_{\ga}\bB'}^{op}$$
which is defined by $\widehat{\ga}^{op}(op(\alp)):= op(\ga(\alp))$. Hence, the following is a commutative diagram of Grothendieck transformations:
\[
\xymatrix
{
\bB \ar[rr]^{\ga} \ar[d]_{\widehat{op}} && \op{Im}_{\ga}\bB' \ar[d]^{\widehat{op}} \\
\widehat{\bB}^{op} \ar[rr]_{\widehat{\ga}^{op}} && \widehat{\op{Im}_{\ga}\bB'}^{op}. 
}
\]
\end{pro}
In the above proposition we need the fact that 
$\ga:\bB_* \to (\op{Im}_{\ga}\bB')_* = \op{Image}(\ga:\bB_* \to \bB'_*)$ is surjective. So, we introduce the following notions.
\begin{defn}\label{co-surj} Let $\ga:\bB \to \bB'$ be a Grothendieck transformation. 
\begin{enumerate}
\item If the covariant functor part $\ga:\bB_* \to \bB'_*$ is surjective, i.e., $\ga:\bB_*(X) \to \bB'_*(X)$ is surjective for any $X$, then $\ga:\bB \to \bB'$ shall be called \emph{a ``covariant-surjective'' Grothendieck transformation}, abusing words.  
\item (for later use) If the contravariant functor part $\ga:\bB^* \to \bB'^*$ is surjective, i.e., $\ga:\bB^*(X) \to \bB'^*(X)$ is surjective for any $X$, then $\ga:\bB \to \bB'$ shall be called \emph{a ``contravariant-surjective'' Grothendieck transformation}, abusing words. 
\item If  $\ga:\bB_* \to \bB'_*$ and  $\ga:\bB^* \to \bB'^*$ are both surjective, $\ga:\bB \to \bB'$ shall be called \emph{a ``covariant+contravariant-surjective"} or \emph{``bi-surjective" Grothendieck transformation}.
\end{enumerate}
\end{defn}
\begin{cor}\label{cor-co-surj} Let $\ga:\bB \to \bB'$ be a ``covariant-surjective'' Grothendieck transformation. Then there is a canonical Grothendieck transformation
$$\widehat{\ga}^{op}: \widehat{\bB}^{op} \to \widehat{\bB'}^{op}$$
which is defined by $\widehat{\ga}^{op}(op(\alp)):= op(\ga(\alp)).$
\end{cor}
\begin{rem}In the above we use the fact that the covariant functor $\ga:\bB_* \to \bB'_*$ is surjective. So, instead of considering the surjective Grothendieck transformation $\ga:\bB \to \op{Im}_{\ga}\bB'$ in order to obtain $\widehat{\ga}^{op}: \widehat{\bB}^{op} \to \widehat{\op{Im}_{\ga}\bB'}^{op}$, one might be tempted to consider the following \emph{slightly modified} one:
$${}_*\bB'(X \xrightarrow f Y) = \begin{cases}
\op{Image}(\ga:\bB_*(X) \to \bB'_*(X)), & \text {if $f:X \to pt$}\\
\bB'(X \xrightarrow f Y),  & \text {otheriwse}
\end{cases}
$$
by which we would obtain $\widehat{\ga}^{op}: \widehat{\bB}^{op} \to \widehat{{}_*\bB'}^{op}$. Unfortunately ${}_*\bB'$ may not be a bivariant theory, to be more precise, the bivariant product may not be well-defined, although the pushforward and the pullback are well-defined.
\end{rem}
\section{Co-operational bivariant theory}\label{coop}
In this section, motivated by the above Fulton--MacPherson's operational bivariant theory, we consider another kind of operational bivariant theory associated to \emph{a contravariant functor} which may not be a multiplicative cohomology theory. This new bivariant theory shall be called \emph{a co-operational bivariant theory}.

If a contravariant functor $h^*$ is a multiplicative cohomology theory and $t: h^* \to \widetilde h^*$ is a natural transformation between two multiplicative cohomology theories, then, as remarked in Remark \ref{contra}, one can extend $t:h^* \to \widetilde h^*$ to a Grothendieck transformation $t:h^*(X \to Y) \to \widetilde h^*(X \to Y)$ between their associated bivariant theories in such a way that for an identity map $id_X:X \to X$, $t:h^*(X \xrightarrow {\op{id}_X} X) \to \widetilde h^*(X \xrightarrow {\op{id}_X} X)$ is equal to the original 
homomorphism $t:h^*(X) \to \widetilde h^*(X)$. If it is not multiplicative, one cannot make such a construction. For example, the Chern class $c: K^0(-) \to H^*(-;\bZ)$ is not a natural transformation between two multiplicative cohomology theories, as proved by B. Totaro \cite{Totaro} (cf. \cite{Groth} and \cite{Segal}).
Our co-operational bivariant theory is defined as follows:
\begin{thm}\label{coop-biv}(``Co-operational" bivariant theory)
Let $F^*(-)$ be a contravariant functor with values in the category of graded abelian groups, defined on a category $C$ with all fiber squares independent.
For each  
map $f:X \to Y$ an element
$$ c \in \mathbb B^{coop}F^i(X \xrightarrow f Y)$$
is defined to be a collection of 
homomorphisms
$$c_g: F^m(X') \to F^{m+i}(Y')$$
for all $m \in \mathbb Z$, all $g:Y' \to Y$  and the fiber square \, \, \, 
$$\CD
X' @> {g'}  >> X\\
@V {f'} VV @VV fV\\
Y' @>> {g} > Y.\\
\endCD
$$ \, \, \, 

And these homomorphisms $c_g$ are
required to be compatible with pullback, i.e., for a fiber diagram
\begin{equation}\label{cd-0}
\CD
X'' @> {h'}  >> X' @> {g'}  >> X\\
@V {f''} VV @VV f'V @VV fV\\
Y'' @>> {h} > Y' @>> {g} > Y.\\
\endCD
\end{equation}
the following diagram commutes
\begin{equation}\label{cd-01}
\xymatrix
{F^m(X'') \ar[d]_{c_{g \circ h}} && F^m(X')\ar[ll]_{(h')^*} \ar[d]^{c_g}\\
F^{m+i}(Y'') && F^{m+i}(Y') \ar[ll]^{h^*}.
}
\end{equation}
We define the following bivariant-theoretic operations:
\begin{enumerate}
\item {\bf Product}: The product
$$\bullet: \mathbb B^{coop}F^i(X \xrightarrow f Y) \otimes \mathbb B^{coop}F^j(Y \xrightarrow g Z) \to \mathbb B^{coop}F^{i+j}(X \xrightarrow {g \circ f} Z)$$
is defined by, for $c \in \mathbb B^{coop}F^i(X \xrightarrow f Y) $ and $d \in \mathbb B^{coop}F^j(Y \xrightarrow g Z) $,
$$(c \bullet d)_h:= d_h \circ c_{h'}: F^m(X') \xrightarrow {c_{h'}}  F^{m+i}(Y') \xrightarrow {d_h} F^{m+i+j}(Z')$$
Here we consider the following fiber squares:
\begin{equation}\label{cd-1}
\CD
X' @> {h''} >> X \\
@V {f'}VV @VV {f}V\\
Y' @> {h'} >> Y \\
@V {g'}VV @VV {g}V \\
Z'  @>> {h} > Z \endCD 
\end{equation}
\item {\bf Pushforward}: For $X \xrightarrow f Y \xrightarrow g Z$ 
$$f_*: \mathbb B^{coop}F^i(X \xrightarrow {g \circ f} Z) \to \mathbb B^{coop}F^i(Y \xrightarrow g Z)$$
is defined by, for $c \in \mathbb B^{coop}F^i(X \xrightarrow {g \circ f} Z)$ 
$$(f_*c)_h := c_h \circ (f')^*: F^m(Y') \xrightarrow{(f')^*} F^m(X')  \xrightarrow{c_h} F^{m+i}(Z').$$
Here we use the above commutative diagram (\ref{cd-1}).
\item {\bf Pullback}: For a fiber square 
$$\CD
X' @> {g'}  >> X\\
@V {f'} VV @VV fV\\
Y' @>> {g} > Y\\
\endCD
$$
$$g^*: \mathbb B^{coop}F^i(X \xrightarrow f Y ) \to \mathbb B^{coop}F^i(X' \xrightarrow {f'} Y' )$$
is defined by, for $c \in \mathbb B^{coop}F^i(X \xrightarrow f Y) $
$$(g^*c)_h := c_{g \circ h}: F^m(X'') \to F^{m+i}(Y'').$$
Here we use the above commutative diagrams (\ref{cd-0}).
\end{enumerate} 
Then 
$\mathbb B^{coop}F^*$
becomes a bivariant theory, i.e., satisfies the seven axioms of bivariant theory, listed in \S \ref{BT}.
\end{thm}
\begin{proof} The proofs 
could be omitted. However, since the definition of our co-operational bivariant theory is different from that of Fulton--MacPherson's operational bivariant theory, for the sake of convenience or time-saving for the reader, we write down proofs, giving just comments for trivial ones. 
\begin{enumerate}
\item First we remark that these three operations are compatible with pullbacks. Here we show only the compatibility of the pushforward. For which we consider the following fiber squares:
\begin{equation}\label{cd-1-3}
\CD
X'' @> {k''} >>X' @> {h''} >> X \\
@V {f''}VV @V {f'}VV @VV {f}V\\
Y'' @> {k'} >> Y' @> {h'} >> Y \\
@V {g''}VV @V {g'}VV @VV {g}V \\
Z'' @>> {k} >Z'  @>> {h} > Z \endCD 
\end{equation}
and consider the following diagrams:
\begin{equation}
\xymatrix
{F^m(X'') \ar[dd]_{c_{h\circ k}} &&&& F^m(X') \ar[llll]_{(k'')^*} \ar[dd]^{c_h}\\
& F^m(Y'') \ar[lu]_{(f'')^*} \ar[dl]^{(f'')_*c_{h \circ k}}  && F^m(Y')\ar[ll]_{(k')^*} \ar[dr]_{(f')_*c_h} \ar[ur]^{(f')^*} &\\
F^{m+i}(Z'') &&&& F^{m+i}(Z') \ar[llll]^{k^*}
}
\end{equation}
The outer square, the upper trapezoid, the left and right triangles are commutative, therefore the lower trapezoid is commutative, which means that the pushforward  is compatible with pullback.
\item {\bf $(A_1)$ Product is associative}: It is because the product is defined as composition of homomorphisms and composition is associative.
\item  {\bf $(A_2)$ Pushforward is functorial}: for
\xymatrix
{
X  \ar@/^10pt/[rrr]^{\maru{$c$}}  \ar[r]_f & Y \ar[r]_g  & Z \ar[r]_h  & W
}
$$(g\circ f)_* c= (g_* \circ f_*)c, \quad c \in \mathbb B^{coop}F^i(X \xrightarrow {h \circ g \circ f} Y).$$
Consider the following fiber squares:
\begin{equation}\label{cd-2}
\CD
X' @> {j'''} >> X \\
@V {f'}VV @VV {f}V\\
Y' @> {j''} >> Y \\
@V {g'}VV @VV {g}V \\
Z'  @> {j'} >> Z \\
@V {g'}VV @VV {g}V \\
W'  @>> {j} > W 
\endCD 
\end{equation}
Then by the definition we have
\begin{align*}
((g\circ f)_*c)_j & := c_j \circ (g'\circ f')^* , \quad c_j \in F^m(X') \to F^{m+i}(W') \\
                  & = c_j \circ ((f')^* \circ (g')^*) \\
                  & = (c_j \circ (f')^*) \circ (g')^*\\ 
                  & = (f_*c)_j \circ (g')^*\\  
                  & = (g_*(f_*c))_j \\
                  & = ((g_* \circ f_*)c)_j: F^m(Z') \to  F^{m+i}(W')          
\end{align*}
Therefore $(g\circ f)_* c= (g_* \circ f_*)c \in \mathbb B^{coop}F^i(Z \xrightarrow g W).$
\item {\bf $(A_3)$ Pullback is functorial}: It is clear by the definition of pullback.
\item {\bf $(A_{12})$ Product and pushforward commute}: Let us consider the above diagrams (\ref{cd-2}): Let $\alp \in 
\mathbb B^{coop}F^i(X \xrightarrow {g \circ f} Y)$ and $\beta \in \mathbb B^{coop}F^j(Z \xrightarrow h W).$
\begin{align*}
\left ( f_*(\alp \bullet \beta)\right)_j & = (\alp \bullet \beta)_j \circ (f')^*\\
& = (\beta_j \circ \alp_{j'} ) \circ (f')^*\\
& = \beta_j \circ (\alp_{j'} \circ (f')^*)\\
& = \beta_j \circ (f_*\alp)_{j'}\\
& = ((f_*\alp) \bullet \beta)_j
\end{align*}
Hence we have $f_*(\alp \bullet \beta) = (f_*\alp) \bullet \beta.$
\item {\bf $(A_{13})$ Product and pullback commute}: Let us consider the above diagrams (\ref{cd-1-3}): Let $\alp \in 
\mathbb B^{coop}F^i(X \xrightarrow f Y)$ and $\beta \in \mathbb B^{coop}F^j(Y \xrightarrow g Z).$
\begin{align*}
\left ( h^*(\alp \bullet \beta)\right)_k & = (\alp \bullet \beta)_{h \circ k}\\
& = \beta_{h \circ k} \circ \alp_{h' \circ k'}\\
& = (h^*\beta)_k \circ ((h')^*\alp)_{k'}\\
& = ((h')^*\alp) \bullet h^*\beta)_k
\end{align*}
Hence we have $h^*(\alp \bullet \beta) = (h')^*\alp \bullet h^*\beta.$
\item {\bf $(A_{23})$ Pushforward and pullback commute}: Let us consider the above diagrams (\ref{cd-1-3}): Let $\alp \in 
\mathbb B^{coop}F^i(X \xrightarrow {g \circ f} Z)$.
\begin{align*}
\left ((f')_*(h^*\alp) \right)_k & = (h^*\alp)_k \circ (f'')^*\\
& = \alp_{h \circ k} \circ (f'')^*\\
& = (f_*\alp)_{h \circ k} \\
& = (h^*(f_*\alp))_k
\end{align*}
Hence we have $(f')_*(h^*\alp)= h^*(f_*\alp).$
\item {\bf $(A_{123})$ Projection formula}: 
Let us consider the following diagrams:
\begin{equation}
\xymatrix
{
& X'\ar[rr]^{\qquad \qquad g'} \ar'[d][dd]_{f'} && X \ar[dd]^f &&\\
\widetilde {X'} \ar[ur]^{\widetilde {k''}} \ar[dd]_{\widetilde{f'}}  \ar[rr]^{\qquad \qquad \widetilde {g'}} && \widetilde X \ar[dd] \ar[ur]_{\widetilde {k'}} & & &\\
& Y' \ar'[r][rr]^g && Y \ar[rr]^h && Z\\
\widetilde {Y'} \ar[ur]^{k''} \ar[rr]_{\widetilde g} && \widetilde Y \ar[ur]_{k'}^{\widetilde f \quad } \ar[rr]_{\widetilde h} && \tilde Z \ar[ur]_k
}
\end{equation}
Let $\alp \in  \mathbb B^{coop}F^i(X \xrightarrow f Y)$ and $\beta \in  \mathbb B^{coop}F^j(Y' \xrightarrow {h \circ g} Z)$
\begin{align*}
\left ((g')_*(g^*\alp \bullet \beta ) \right)_k & = (g^*\alp \bullet \beta )_k \circ (\widetilde {g'})^*\\
& = (\beta_k \circ (g^*\alp)_{k''} ) \circ (\widetilde {g'})^*\\
& = \beta_k \circ  ((g^*\alp)_{k''} \circ (\widetilde {g'})^* )\\
& = \beta_k \circ  ((\widetilde g)^* \circ \alp_{k'} )\\
& = ( \beta_k \circ  (\widetilde g)^*) \circ \alp_{k'} \\
& = (g_*\beta)_k \circ \alp_{k'} \\
& = (\alp \bullet g_*\beta)_k
\end{align*}
Hence we have $(g')_*(g^*\alp \bullet \beta )= \alp \bullet g_*\beta.$
\end{enumerate}
\end{proof}
\begin{rem} In Fulton--MacPherson's bivariant theory $\bB^*$ the bivariant pushforward $f_*:\bB^*(X \xrightarrow {g \circ f} Z) \to \bB^*(Y \xrightarrow g Z)$ is defined for a confined map $f:X \to Y$. As one can see, in our co-operational bivariant theory, for the bivariant pushforward $f_*: \mathbb B^{coop}F^i(X \xrightarrow {g \circ f} Z) \to \mathbb B^{coop}F^i(Y \xrightarrow g Z)$ we do not need the confined-ness of the map $f:X \to Y$.
\end{rem}
\begin{rem} Let $\mathbb B^*(X \to Y)$ be a bivariant theory and let $c \in \mathbb B^i(X \xrightarrow f Y)$.
Then for any morphism $g:Y' \to Y$ and consider the fiber square
$$\CD
X' @> {g'}  >> X\\
@V {f'} VV @VV fV\\
Y' @>> {g} > Y.\\
\endCD
$$
Then $(g^*c)^*: B_m(Y'):=\mathbb B^{-m}(Y' \to pt) \to  \mathbb B^{-m+i}(X' \to pt) =:B_{m-i}(X')$
is defined by $(g^*c)^*(a):=g^*c \bullet a$ (as defined in \cite[\S 2.5 Gysin homomorphisms]{FM}). For fiber squares
$$\CD
X'' @> {h'}  >> X' @> {g'}  >> X\\
@V {f''} VV @VV f'V @VV fV\\
Y'' @>> {h} > Y' @>> {g} > Y.\\
\endCD
$$ \, \, \, 
the following diagram commutes:
$$
\xymatrix
{B_{m-i}(X'') \ar[rr]^{(h')_*} && B_{m-i}(X') \\
 B_m(Y'') \ar[u]^{((g \circ h)^*c)^*}  \ar[rr]_{h_*} && B_m(Y') \ar[u]_{(g^*c)^*}.
}
$$ 
Then considering the collection $\{ g^*c \, \, | \, \, g:Y' \to Y\}$ gives us an operational bivariant 
class from the given bivariant
class. So, conversely Fulton--MacPherson's operational bivariant theory is one obtained from a covariant functor. Similarly, our co-operational bivariant theory is motivated by considering the following:
$(g^*c)_*: B^m(X'):=\mathbb B^m(X' \xrightarrow {id} X') \to \mathbb B^{m+i}(Y' \xrightarrow {id} Y') =: B^{m+i}(Y')$
is defined by $(g^*c)_*(b):=f'_*(b \bullet g^*c)$ (as defined in \cite[\S 2.5 Gysin homomorphisms]{FM}). Here we remark that, as pointed out in the first paragraph of \cite[\S 2.5 Gysin homomorphisms, p.26]{FM}, when we consider the Gysin homomorphism $\theta_*$ induced by a bivariant element $\theta \in \bB (X \xrightarrow f Y)$ the map $f$ is assumed to be confined. For fiber squares
$$\CD
X'' @> {h'}  >> X' @> {g'}  >> X\\
@V {f''} VV @VV f'V @VV fV\\
Y'' @>> {h} > Y' @>> {g} > Y.\\
\endCD
$$ \, \, \, 
the following diagram commutes:
$$
\xymatrix
{B^m(X'') \ar[d]_{((g\circ h)^*c)_*} && B^m(X')\ar[ll]_{(h')^*} \ar[d]^{(g^*c)_*}\\
 B^{m+i}(Y'') && B^{m+i}(Y') \ar[ll]^{h^*}.
}
$$  
\end{rem}
The following is a \emph{co-operational version} of Proposition \ref{map op}, i.e., of the Grothendieck transformation $op: \bB \to \bB^{op}$. Since the proof of it is similar to that of the proof of Proposition \ref{map op}, it is omitted, left for the reader.

\begin{pro} Let $\mathbb B$ be a bivariant theory. Let $B^*(X):=\mathbb B^*(X \xrightarrow {id_X} X)$ be the associated contravariant functor and let $\bB^{coop}(X \to Y)$ be the co-operational bivariant theory defined by $\bB^{coop}(X \to Y) :=\bB^{coop}B^*(X \to Y)$ as above. Then, restricted to confined\footnote{As remarked above, the reason why we need to consider confined maps is that $(g^*\alp)_*$ is defined by $(g^*\alp)_*(b)= f'_*(b \bullet g^*\alp)$, namely we use the pushforward $f'_*$, which is defined for a confined map.} morphisms $f:X \to Y$, there exists a canonical Grothendieck transformation
$$coop: \mathbb B \to \mathbb B^{coop}$$
which is defined by
$${coop} (\alp) := \left \{ (g^*\alp)_*:B^*(X') \to B^*(Y') \, | \, \, g:Y' \to Y \right \}$$
Here $(g^*\alp)_*$ is defined by $(g^*\alp)_*(b)= f'_*(b \bullet g^*\alp)$, where $b \in B^*(X')=\mathbb B^*(X' \xrightarrow {id_{X'}} X')$ and $g^*\alp \in \mathbb B^*(X' \xrightarrow {f'} Y')$ and we consider the following fiber square
$$\CD
X' @> {g'}  >> X\\
@V {f'} VV @VV fV\\
Y' @>> {g} > Y.\\
\endCD
$$
Then we have
\begin{enumerate}
\item $coop(\alp \bullet \be)= coop(\alp) \bullet coop(\be)$.
\item $coop(f_*\alp)= f_*coop(\alp)$.
\item $coop(f^*\alp)= f^*coop(\alp)$.
\end{enumerate}
\end{pro}
We have the following co-operational versions of Proposition \ref{pro-op}, Proposition \ref{im-ga} 
and Corollary \ref{cor-co-surj}, respectively.
\begin{pro}\label{pro-coop} Let $\bB$ be a bivariant theory and $\bB^{coop}$ be its associated co-operational bivariant theory.
\begin{enumerate}
\item Define
$\widehat{\bB}^{coop}:= \op{Image}(coop:\bB \to \bB^{coop}),$
that is 
$$\widehat{\bB}^{coop}(X \to Y) := coop(\bB(X \to Y))  = \left \{coop(\alp) \, \, | \, \, \alp \in \bB(X \to Y) \right \} \subset \bB^{coop}(X \to Y).$$
Then $\widehat{\bB}^{coop}$ is a co-operational bivariant theory, which is a subtheory of $\bB^{coop}$, and 
$\widehat{coop}: \bB \to \widehat{\bB}^{coop}$ defined by $\widehat{coop}(\alp) = coop(\alp)$ and the inclusion 
$\iota^{coop}_{\bB}:\widehat{\bB}^{coop} \hookrightarrow \bB^{coop}$ are Grothendieck transformations.
\item \label{X-to-X-coop} For the identity map $\op{id}_X:X \to X$ we have the isomorphism:
$$\bB^*(X) = \bB(X \xrightarrow {\op{id}_X} X) \cong \widehat{\bB}^{coop} (X \xrightarrow {\op{id}_X} X).$$
\end{enumerate}
\end{pro}
\begin{proof}
 As in the proof of Proposition \ref{pro-op}, we show only (\ref{X-to-X-coop}).  Let $\alp, \be \in \bB(X \xrightarrow {\op{id}_X} X)$. It suffices to show that $coop(\alp) =coop(\be)$ implies $\alp = \be$, i.e., $(g^*\alp)_* =(g^*\be)_*$ for all $g:Y \to X$ implies that $\alp = \be$. Let us consider the case $\op{id}_X:X \to X$. Then we are supposed to have $(\op{id}_X^*\alp)_* =(\op{id}_X^*\be)_*$, i.e., $(\op{id}_X^*\alp)_* (x) =(\op{id}_X^*\be)_*(x)$, i.e.,
$$x \bullet \alp = (\op{id}_X^*\alp)_* (x) = (\op{id}_X^*\be)_*(x) = x \bullet \be $$
 for any $x \in \bB^*(X)=\bB(X \xrightarrow {\op{id}_X} X)$. In particular, let $x$ be the unit $1_X \in \bB^0(X)=\bB^0(X \xrightarrow {\op{id}_X} X)$, then we have
 $$ \alp = 1_X \bullet \alp = 1_X \bullet \be = \be.$$
 Here the left and right equalities are due to the property of the unit (see ({\bf Units}) in \S 2, or \cite[\S 2.2 Axioms for a bivariant theory, p. 22]{FM}.) 
 Therefore we have that $coop(\alp) =coop(\be)$ implies $\alp = \be$.
 \end{proof} 
For each $\alp \in \bB(X \to Y)$ and each $g:Y' \to Y$ we have the following commutative diagram:
\begin{equation}\label{indiv-com-co}
\xymatrix
{
\bB^*(X') \ar[d]_{(g^*\alp)_*} \ar[rr]^{\ga}  && \bB'^*(X') \ar[d]^{(g^*\ga(\alp))_*} \\
\bB^*(Y') \ar[rr]_{\ga}  && \bB'^*(Y') 
}
\end{equation}
Here we emphasize that $\widehat{coop}(\alp) = coop(\alp) = \{ (g^*\alp)_* \, | \, g:Y' \to Y\}$ and $\widehat{coop}(\ga(\alp)) = coop(\ga(\alp)) = \{ (g^*\ga(\alp))_* \, | \, g:Y' \to Y\}$. Just like the operational bivariant theory, $coop(\alp) =coop(\beta)$ does not necessarily implies that 
$coop(\ga(\alp)) =coop(\ga(\beta))$. However, if $\ga:\bB^* \to \bB'^*$ is surjective, then $coop(\alp) =coop(\beta)$ \emph{does} imply that 
$coop(\ga(\alp)) =coop(\ga(\beta))$. Therefore we have the following:
\begin{pro}\label{im-t}
Let $\ga: \bB \to \bB'$ be a Grothendieck transformation. Then there is a canonical Grothendieck transformation
$$\widehat{\ga}^{coop}: \widehat{\bB}^{coop} \to \widehat{\op{Im}_{\ga}\bB'}^{coop}$$
which is defined by $\widehat{\ga}^{coop}(coop(\alp)):= coop(\ga(\alp))$. Hence, the following is a commutative diagram of Grothendieck transformations:
\[
\xymatrix
{
\bB \ar[rr]^{\ga} \ar[d]_{\widehat{coop}} && \op{Im}_{\ga}\bB' \ar[d]^{\widehat{coop}} \\
\widehat{\bB}^{coop} \ar[rr]_{\widehat{\ga}^{coop}} && \widehat{\op{Im}_{\ga}\bB'}^{coop} 
}
\]
\end{pro}
In the above proposition we need the fact that $\ga:\bB^* \to (\op{Im}_{\ga}\bB')^* = \op{Image}(\ga: \bB^* \to \bB'^*)$ is surjective. So we have the following:
\begin{cor}
Let $\ga:\bB \to \bB'$ be a ``contravariant-surjective'' Grothendieck transformation. Then there is a canonical Grothendieck transformation
$$\widehat{\ga}^{coop}: \widehat{\bB}^{coop} \to \widehat{\bB'}^{coop}$$
which is defined by $\widehat{\ga}^{coop}(coop(\alp)):= coop(\ga(\alp))$
\end{cor}
Unlike the case of a natural transformation $t:h_* \to \tilde h_*$ of covariant functors (see the paragraph right after Remark \ref{3-rem-bb}), 
as in Remark \ref{contra} (see \cite[\S 3.2 ]{FM}), there exists a Grothendieck transformation of their associated bivariant theories
$$t:h^*(X \to Y) \to h'^*(X \to Y), \text{\, which shall be denoted by} \quad t:\bB h(X \to Y) \to \bB h'(X \to Y)$$
such that for the identity map $\op{id}_X: X \to X$ the homomorphism $t:\bB h(X \xrightarrow {\op{id}_X} X) \to \bB h'(X \xrightarrow {\op{id}_X} X) $ is equal to the original homomorphism $t: h^*(X) \to h'^*(X)$ (see \cite[\S 3.1.7]{FM}). 
\begin{thm}\label{thh'} Let $t: h^* \to h'^*$ be a natural transformation of two multiplicative cohomology theories. 
\begin{enumerate}
\item Then there exits a Grothendieck transformation of co-operational bivariant theories
$$\widehat t^{coop}: \widehat{\bB h}^{coop} \to \widehat{\op{Im}_t\bB h'}^{coop}$$
such that for the identity map $id_X:X \to X$ the following composition of homomorphisms
$$e_{\bB h'} \circ \widehat t^{coop}:\widehat{\bB h}^{coop} (X \xrightarrow {\op{id}_X} X)  \to \widehat{\op{Im}_t \bB h'}^{coop}(X \xrightarrow {\op{id}_X} X) \hookrightarrow \widehat{\bB h'}^{coop} (X \xrightarrow {\op{id}_X} X) $$
 is equal to the original homomorphism  $t: h^*(X) \to h'^*(X)$ 
via the isomorphisms 
$$\widehat{\bB h}^{coop} (X \xrightarrow {\op{id}_X} X) \cong h^*(X), \widehat{\bB h'}^{coop} (X \xrightarrow {\op{id}_X} X) \cong h'^*(X): coop(\alp) \leftrightarrow \alp.$$ 
Here $e_{\bB h'} : \widehat{\op{Im}_t \bB h'}^{coop}(X \xrightarrow {\op{id}_X} X) \hookrightarrow \widehat{\bB h'}^{coop} (X \xrightarrow {\op{id}_X} X) $ is the inclusion.
\item For each $\alp \in h^*(X)$ and each $g:X' \to X$, the homomorphism $(g^*\alp)_*$, an element of $coop(\alp)$, is the homomorphism defined by taking the cup-product with $g^*\alp$:
$$ (-)\cup g^*\alp: h^*(X') \to h^*(X'), \text{i.e., defined by $((-) \cup g^*\alp)(x):= x \cup g^*\alp$ for $x \in h^*(X')$}$$
and the following diagram commutes:
\begin{equation}\label{cd-cup}
\xymatrix
{
h^*(X') \ar[rr]^t \ar[d]_{(-) \cup g^*\alp} && h'^*(X') \ar[d]^{(-) \cup g^*(t(\alp))} \\
h^*(X')\ar[rr]_{t} && h'^*(X').
}
\end{equation}
\end{enumerate}
 \end{thm}
\begin{proof} 

(1) It follows from the above Proposition \ref{im-t}, applying it to the Grothendieck transformation $t:\bB h(X \to Y) \to \bB h'(X \to Y)$, that there exists a Grothendieck transformation of co-operational bivariant theories
$$\widehat t^{coop}: \widehat{\bB h}^{coop} \to \widehat{\op{Im}_t\bB h'}^{coop}$$
Now by the above construction we have
$$\widehat{\bB h}^{coop} (X \xrightarrow {\op{id}_X} X) = \{ coop(\alp) \, | \, \alp \in \bB h(X \xrightarrow {\op{id}_X} X) =h^*(X) \}.$$
Since $\widehat t^{coop}(coop(\alp)) =coop(t(\alp))$, 
$$e_{\bB h'} \circ \widehat t^{coop}:\widehat{\bB h}^{coop} (X \xrightarrow {\op{id}_X} X)  \to \widehat{\op{Im}_t \bB h'}^{coop}(X \xrightarrow {\op{id}_X} X) \hookrightarrow \widehat{\bB h'}^{coop} (X \xrightarrow {\op{id}_X} X) $$
 is equal to the original homomorphism  $t: h^*(X) \to h'^*(X)$ 
via the isomorphism 
$\widehat{\bB h}^{coop} (X \xrightarrow {\op{id}_X} X) \cong h^*(X): coop(\alp) \leftrightarrow \alp.$\\

(2) First we note that for $ \alp \in \bB h(X \xrightarrow {\op{id}_X} X) =h^*(X)$ we have
$$coop(\alp)= \{(g^*\alp)_* : h^*(X') \to h^*(X') | g:X' \to X \}.$$
Here we emphasize that the homomorphism $(g^*\alp)_* : h^*(X') \to h^*(X')$ is defined by, for $x' \in h^*(X')$
$$(g^*\alp)_*(x')=(\op{id}_{X'})_*(x' \bullet g^*\alp) = x' \bullet g^*\alp,$$
which is equal to the cup-product $x' \cup g^*\alp$, by the definition of the associated bivariant theory $\bB h^*(X \to Y)$ constructed from the multiplicative cohomology theory $h^*$ (see \cite[\S 3.1.4 - \S 3.1.9]{FM}, in particular \cite[\S 3.1.7]{FM} for $x' \bullet g^*\alp$ being equal to the cup-product $x' \cup g^*\alp$.)  Hence the homomorphism $(g^*\alp)_*$ is one defined by taking the cup-product with $g^*\alp$ and the commutativity of the above diagram follows from the property of the cup-product.

\end{proof}
\begin{rem} The diagram (\ref{cd-cup}) means that the above homomorphisms $ (-) \cup g^*\alp: h^*(X') \to h^*(X')$ and  $ (-) \cup g^*(t(\alp)): h'^*(X') \to h'^*(X')$, i.e., cohomology operations defined by taking the cup-product with $g^*\alp$ and $g^*(t(\alp))$, are compatible with the natural transformation $t:h^* \to h'^*$. Speaking of the cup-product, the following $k$-times cup-product is also \emph{a cohomology operation}:
$$\text{$\sqcup^k:h^*(X) \to h^*(X)$ defined by $\sqcup^k(x):=x^k = \underbrace{x \cup x \cup \cdots \cup x}_k$,}$$
and it is also compatible with the natural transformation $t:h^* \to h'^*$.
To be more precise, if we write the degree, we have $\sqcup^k:h^m(X) \to h^{km}(X)$.
And the following diagram commutes:
\begin{equation}\label{sq-cup}
\xymatrix
{
h^*(X') \ar[rr]^t \ar[d]_{\sqcup^k} && h'^*(X') \ar[d]^{\sqcup^k} \\
h^*(X')\ar[rr]_{t} && h'^*(X').
}
\end{equation}
Because $t(\sqcup^k(x)) =t(x^k)= t(\underbrace{x \cup \cdots \cup x}_k) = \underbrace{t(x) \cup \cdots \cup t(x)}_k =(t(x))^k =\sqcup^k(t(x))$,
So let us denote 
$\sqcup^k:=\{\sqcup^k:h^*(X') \to h^*(X') \, \, | \, \, g:X' \to X \}.$
Clearly $\sqcup^k \not \in \widehat{\bB}^{coop}h(X \xrightarrow {\op{id}_X} X)$, for $k=2, 3, \cdots$.
In fact, $\sqcup^k:h^*(X) \to h^*(X)$ is \emph{not a homomorphism}, hence $\sqcup^k \not \in \bB^{coop}h(X \xrightarrow {\op{id}_X} X)$
because we define a co-operational bivariant class to be a collection of homomorphisms. So, if we consider a collection of \emph{maps} instead of \emph{homomorphisms}, then we get another larger co-operational bivariant theory $\overline{\bB}^{coop}h(X \to Y)$, then we can have 
$\sqcup^k \in \overline{\bB}^{coop}h(X \xrightarrow {\op{id}_X} X)$. 
\end{rem}
The above theorem means the following:
\begin{enumerate}
\item In the case of a natural transformation $t:h^* \to h'^*$ between \emph{multiplicative cohomology theories} $h^*$ and $h'^*$, \emph{using the associated bivariant theories $\bB h$ and $\bB h'$ constructed by Fulton and MacPherson} \cite[\S 3.1 Construction of a bivariant theory from a cohomology theory]{FM}, we can extend the natural transformation $t:h^* \to h'^*$ to a Grothendieck transformation to co-operational bivariant theories
$$\widehat t^{coop}: \widehat{\bB h}^{coop} \to \widehat{\op{Im}_t \bB h'}^{coop}$$ 
by restricting to the subtheories $\widehat{\bB h}^{coop} \subset \bB h^{coop}$ and $\widehat{\op{Im}_t \bB h'}^{coop} \subset \bB h'^{coop}$, although we may not be able to extend it to the whole co-operational bivariant theories $\bB h^{coop}$ and $\bB h'^{coop}$.
\item Each element $c_g$ of a co-operational bivariant class $c =\{c_g :h^*(X') \to h^*(X') \, | \, g:X' \to X\}\in \widehat{\bB h}^{coop}(X \xrightarrow {\op{id}_X} X)$ 
is a cohomology operation of $h^*$, which is \emph{a simple homomorphism defined by taking the cup-product} with $g^*\alp$ for $\alp \in h^*(X)$, i.e.,
\begin{equation}\label{c_g}
c_g(x) = x \cup g^*\alp.
\end{equation}
Since we have the isomorphism $\widehat{\bB h}^{coop}(X \xrightarrow {\op{id}_X} X) \cong h^*(X),\, coop(\alp) \longleftrightarrow \alp$, we can denote the co-operational bivariant class $c$ by $\alp$, thus the above equality (\ref{c_g}) could be expressed by, abusing notation or symbol,
\begin{equation*}
\alp_g(x) = x \cup g^*\alp.
\end{equation*}
\item However, by this construction, another co-operational bivariant class $\sqcup^k \in \overline{\bB}^{coop}h(X \xrightarrow {\op{id}_X} X)$ consisting of \emph{very simple cohomology operations} $\sqcup^k:h^*(X') \to h'^*(X)$ for all $g:X' \to X$ \emph{cannot belong to} $\widehat{\overline{\bB} h}^{coop}(X \xrightarrow {\op{id}_X} X)$, since $\widehat{\overline{\bB} h}^{coop}(X \xrightarrow {\op{id}_X} X)$ is a co-operational bivariant theory constructed from a multiplicative cohomology theory $h^*$, thus a co-operational bivariant class $c \in \widehat{\overline{\bB} h}^{coop}(X \xrightarrow {\op{id}_X} X)$ \emph{has to} satisfy the above formula (\ref{c_g}), but clearly $\sqcup^k$ \emph{does not} satisfy (\ref{c_g}).
\end{enumerate}
So, after observing these, we want to pose the following question:
\begin{qu} Can one construct another associated bivariant theory $\widetilde {\bB} h$ (which is different from Fulton--MacPherson's bivariant theory $\bB h$) for a multiplicative cohomology theory $h^*$ such that 
$$\sqcup^k \in \widehat{\widetilde {\bB} h}^{coop}(X \xrightarrow {\op{id}_X} X)?$$
\end{qu}

Suppose that we have natural transformation of two contravariant functors which are not necessarily multiplicative cohomology theories:
$$T:F^* \to G^*.$$
From now on a co-operational bivariant class is a collection of maps, not necessarily homomorphisms, and also the co-operational bivariant theory $\bB^{coop}F^*(X \to Y)$ obtained from a contravariant functor is considered as a set, even if it can have finer algebraic structures.
As we see above that even in the case of a Grothendieck transformation $\ga: \bB \to \bB'$ it is not clear whether there is a Grothendieck transformation between the associated co-operational bivariant theories $\bB^{coop}$ and $\bB'^{coop}$, it is not clear at all whether there exists a Grothendieck transformation
$$\ga_T: \mathbb B^{coop} F^* \to  \mathbb B^{coop} G^* .$$
For that we need to 
consider a subtheory $ \mathbb B^{coop}_T F^*$ of $\mathbb B^{coop} F^*$, which is defined
as follows.

Motivated by (\ref{indiv-com-co}), (\ref{cd-cup}) and  (\ref{sq-cup}), 
we define the following:
\begin{defn}\label{4-defn-bcoop} For $f:X \to Y$,  
an element
$$c \in  \mathbb B^{coop}_T F^i(X \xrightarrow f Y) \quad (\subset \mathbb B^{coop}F^i(X \xrightarrow f Y) )$$
is defined to satisfy
that there exists an element  
$$c^T \in  \mathbb B^{coop} G^i(X \to Y)$$
such that
the following diagram commutes:for $m \in \mathbb Z$ and for $g:Y' \to Y$
\begin{equation}\label{cd-7}
\xymatrix
{F^m(X') \ar[d]_{c_g} \ar[rr]^T && G^m(X') \ar[d]^{(c^T)_g}\\
F^{m+i}(Y') \ar[rr]_T && G^{m+i}(Y').
}
\end{equation}
\end{defn}
\begin{rem}\label{T-coop}
We emphasize that therefore we automatically have the following commutative cube:
\begin{equation}\label{ccube-2}
\xymatrix
{ F^*(X') \ar[dd]_{c_g } \ar[rd]^{h^*} \ar[rr]^{T} && G^*(X') \ar'[d][dd]^(.3){(c^T)_g} \ar[rd]^{h^*} \\
& F^*(X'') \ar[dd]_(.3){ c_{g \circ h}} \ar[rr]^(.4){T}  &&  G^*(X'') \ar[dd]^{(c^T)_{g \circ h}} \\
F^*(X') \ar'[r] [rr]_{T \quad \quad }  \ar[rd]_{h^*} && G^*(X') \ar[rd]_{h^*} \\
& F^*(X'') \ar[rr] _{T \quad \quad  \quad }  &&  G^*(X'').
}
\end{equation}
\end{rem}
\begin{thm}
The above $\mathbb B^{coop}_T F^*$ is a bivariant theory. Namely, the following product, pushforward and pullback are all well-defined.
\begin{enumerate}
\item {\bf Product}:\quad $\bullet: \mathbb B^{coop}_TF^i(X \xrightarrow f Y) \otimes \mathbb B^{coop}_TF^j(Y \xrightarrow g Z) \to \mathbb B^{coop}_TF^{i+j}(X \xrightarrow {g \circ f} Z)$,
\item {\bf Pushforward}: \quad $f_*: \mathbb B^{coop}_TF^*(X \xrightarrow {g \circ f} Z) \to \mathbb B^{coop}_TF^*(Y \xrightarrow g Z)$,
\item {\bf Pullback}: For a fiber square 
$$\CD
X' @> {g'}  >> X\\
@V {f'} VV @VV fV\\
Y' @>> {g} > Y\\
\endCD
$$
$$g^*: \mathbb B^{coop}_TF^*(X \xrightarrow f Y ) \to \mathbb B^{coop}_TF^*(X' \xrightarrow {f'} Y' ).$$
\end{enumerate} 
\end{thm}
\begin{proof}
As stated in the above statement, it suffices to show the well-defined-ness of the above operations, since the properties of the seven axioms of a bivariant theory already hold in $\mathbb B^{coop}F^*$.
\begin{enumerate}
\item {\bf Product}: We show that for $c \in \mathbb B^{coop}_TF^i(X \xrightarrow f Y) $ and $d \in \mathbb B^{coop}_TF^j(Y \xrightarrow g Z) $, the product
$c \bullet d$ belongs to $\mathbb B^{coop}_TF^{i+j}(X \xrightarrow {g \circ f} Z)$, namely there exists some bivariant element $(c \bullet d)^T \in \bB^{coop}G^{i+j}(X \xrightarrow {g \circ f} Z)$ such that the following diagram commutes:
\begin{equation}\label{dia-1}
\xymatrix
{
F^m(X')  \ar[d]_{(c \bullet d)_h} \ar[rr]^{T} && G^m(X') \ar[d]^{((c \bullet d)^T)_h}\\
F^{m+i+j}(Z') \ar[rr]_{T}  && G^{m+i+j}(Z').  \\
}
\end{equation}
Here we consider the following fiber squares:
\begin{equation}\label{cd-1-1}
\CD
X' @> {h''} >> X \\
@V {f'}VV @VV {f}V\\
Y' @> {h'} >> Y \\
@V {g'}VV @VV {g}V \\
Z'  @>> {h} > Z. \endCD 
\end{equation}
In the diagram (\ref{big-dia}) below we have that
\begin{enumerate}
\item the upper and lower trapezoids are both commutative, since $c \in \mathbb B^{coop}_TF^i(X \xrightarrow f Y) $ and $d \in \mathbb B^{coop}_TF^j(Y \xrightarrow g Z) $, 
\item the left and right triangles are both commutative, since by the definition of the co-operational bivariant product $\bullet$ we have
$(c \bullet d)_h = d_h \circ c_{h'}$ and $(c^T \bullet d^T)_h = (d^T)_h \circ (c^T)_{h'}$.
\end{enumerate}
\begin{equation}\label{big-dia}
\xymatrix
{
F^m(X')  \ar[dd]_{(c \bullet d)_h} \ar[dr]^{c_{h'}} \ar[rrr]^{T} &&& G^m(X')  \ar[dl]_{(c^T)_{h'}} \ar[dd]^{(c^T \bullet d^T)_h}\\
& F^{m+i}(Y')  \ar[r]^{T} \ar[dl]^{d_h} & G^{m+i}(Y')  \ar[dr]_{(d^T)_h}\\
F^{m+i+j}(Z') \ar[rrr]_{T}  &&& G^{m+i+j}(Z').  \\
}
\end{equation}
Thus, the outer square of the above diagram (\ref{big-dia}) is commutative. So, if we set 
\begin{equation*}
(c \bullet d)^T := c^T \bullet d^T,
\end{equation*}
then we obtain the above commutative square (\ref{dia-1}). Hence $c \bullet d \in \mathbb B^{coop}_TF^{i+j}(X \xrightarrow {g \circ f} Z)$.
\item {\bf Pushforward}: We show that for $c \in \mathbb B^{coop}_TF^i(X \xrightarrow {g \circ f} Z)$, $f_*c$ belongs to $\mathbb B^{coop}_TF^i(Y \xrightarrow {g} Z)$, i.e., there exists some bivariant element $(f_*c)^T  \in \bB^{coop}G^i(Y \xrightarrow g Z)$ such that the following diagram commutes:
\begin{equation}\label{dia-2}
\xymatrix
{
F^m(Y') \ar[d]_{(f_*c)_h} \ar[rr]^T &&  G^m(Y') \ar[d]^{((f_*c)^T)_h}\\
F^{m+i}(Z')  \ar[rr]_T &&  G^{m+i}(Z'). 
}
\end{equation}
Here we consider the above squares (\ref{cd-1-1}) again.
In the diagram (\ref{big-dia2}) below we have that
\begin{enumerate}
\item the upper trapezoid is commutative, since $T$ is a contravariant functor,
\item the lower trapezoid is commutative, since $c \in \mathbb B^{coop}_TF^i(X \xrightarrow {g \circ f} Z)$,
\item the left and right triangles are both commutative, since by the definition of the pushforward we have $(f_*c)_h = c_h \circ (f')^*$ and $(f_*c^T)_h = (c^T)_h \circ (f')^*$. 
\end{enumerate}
\begin{equation}\label{big-dia2}
\xymatrix
{
F^m(Y') \ar[dd]_{(f_*c)_h} \ar[dr]^{(f')^*} \ar[rrr]^T &&& G^m(Y')  \ar[dl]_{(f')^*} \ar[dd]^{(f_*c^T)_h} \\
& F^m(X') \ar[dl]^{c_h} \ar[r]^T & G^m(X') \ar[dr]_{(c^T)_h}  \\
F^{m+i}(Z')  \ar[rrr]_T &&& G^{m+i}(Z').
}
\end{equation}
Thus, the outer square of the above diagram (\ref{big-dia2}) is commutative. So, if we set 
\begin{equation*}
(f_*c)^T := f_*c^T,
\end{equation*}
then we obtain the above commutative square (\ref{dia-2}). Hence $f_*c \in \mathbb B^{coop}_TF^i(Y \xrightarrow {g} Z)$.
\item {\bf Pullback}:We show that for $c \in \mathbb B^{coop}_TF^i(X \xrightarrow f Z)$, $g^*c$ belongs to $\mathbb B^{coop}_TF^i(X' \xrightarrow {f'} Y')$, i.e., there exists some bivariant element $(g^*c)^T  \in \bB^{coop}G^i(X' \xrightarrow {f'} Y')$ such that the following diagram commutes:
\begin{equation}\label{dia-3}
\xymatrix
{F^m(X'') \ar[d]_{(g^*c)_h} \ar[rr]^T && G^m(X'') \ar[d]^{((g^*c)^T)_h}\\
F^{m+i}(Y'') \ar[rr]_T && G^{m+i}(Y'')
}
\end{equation}
where we consider the following fiber squares:
\begin{equation*}
\CD
X'' @> {h'}  >> X' @> {g'}  >> X\\
@V {f''} VV @VV f'V @VV fV\\
Y'' @>> {h} > Y' @>> {g} > Y.\\
\endCD
\end{equation*}
Since $c \in \mathbb B^{coop}_TF^i(X \xrightarrow f Z)$ and $(g^*c)_h = c_{g \circ h}$ by definition, we have the following commutative diagram:
\begin{equation}\label{cd-4}
\xymatrix
{F^m(X'') \ar[d]_{c_{g \circ h} = (g^*c)_h} \ar[rr]^T && G^m(X'') \ar[d]^{(c^T)_{g \circ h} = (g^*(c^T))_h}\\
F^{m+i}(Y'') \ar[rr]_T && G^{m+i}(Y'').
}
\end{equation}
So, if we set 
\begin{equation*}
(g^*c)^T:=g^*(c^T),
\end{equation*} 
we get the above commutative square (\ref{dia-3}). Thus we get $g^*c \in \bB^{coop}_TF^i(X' \xrightarrow {f'} Y')$.
\end{enumerate}
\end{proof}
\begin{rem}\label{surj} In the proof of the above Theorem \ref{T-coop} we set 
\begin{enumerate}
\item $(c \bullet d)^T := c^T \bullet d^T$,
\item $(f_*c)^T :=f_*(c^T)$, 
\item $(g^*c)^T :=g^*(c^T)$. 
\end{enumerate}
However, we see that if $T:F^* \to G^*$ is surjective, i.e., $T:F^*(W) \to G^*(W)$ is surjective for any object $W$, then these equalities automatically hold.
For example, as to (1), we can show it as follows. We have the following commutative diagrams:
\begin{equation*}\label{dia-6}
\xymatrix
{
F^m(X')  \ar[d]_{(c \bullet d)_h} \ar[rrr]^{T} &&& G^m(X') \pdarrow{((c \bullet d)^T)_h}{(c^T \bullet d^T)_h}\\
F^{m+i+j}(Z') \ar[rrr]_{T}  &&& G^{m+i+j}(Z').  \\
}
\end{equation*}
Hence we have
\begin{equation*}
((c \bullet d)^T)_h \circ T = T \circ (c \bullet d)_h = (c^T \bullet d^T)_h \circ T,
\end{equation*}
which implies that $((c \bullet d)^T)_h \circ T = (c^T \bullet d^T)_h \circ T$, i.e., $\left (((c \bullet d)^T)_h - (c^T \bullet d^T)_h \right) \circ T=0$. So, if $T:F^m(X') \to G^m(X')$ is surjective, then $((c \bullet d)^T)_h - (c^T \bullet d^T)_h =0$, hence $((c \bullet d)^T)_h = (c^T \bullet d^T)_h$. Therefore we get $(c \bullet d)^T = c^T \bullet d^T$. The same argument can be applied to (2) and (3).
\end{rem}
A Grothendieck transformation which we are looking for is the following:
\begin{cor} Let $T:F^* \to G^*$ be a natural transformation between two contravariant functors. For $\mathbb B^{coop}_T F^*$ we assume that bivariant elements $c^T$ in Definition \ref{4-defn-bcoop} satisfy the following three properties (if $T$ is surjective, they are automatically satisfied, as pointed out in Remark \ref{surj}):
\begin{enumerate}
\item $(c \bullet d)^T = c^T \bullet d^T$,
\item $(f_*c)^T =f_*(c^T)$, 
\item $(g^*c)^T =g^*(c^T)$. 
\end{enumerate}
\noindent 
Then we have a Grothendieck transformation
$$\ga_T: \mathbb B^{coop}_T F^* \to  \mathbb B^{coop} G^*$$
which is a collection of 
$\ga_T: \mathbb B^{coop}_T F^*(X \xrightarrow f Y)  \to  \mathbb B^{coop} G^*(X \xrightarrow f Y)$
defined by for each $c \in \mathbb B^{coop}_T F^*(X \to Y)$
$$\ga_T(c) := c^T.$$
Furthermore $\ga_T(c) = c^T$ means that for $m \in \mathbb Z$ and for $g:Y' \to Y$ we have the commutative diagram (\ref{cd-7}), i.e., \begin{equation*}
\xymatrix
{F^m(X') \ar[d]_{c_g} \ar[rr]^T && G^m(X') \ar[d]^{(c^T)_g}\\
F^{m+i}(Y') \ar[rr]_T && G^{m+i}(Y').
}
\end{equation*}
\end{cor}
\section{$ \mathbb B^{coop}F^{*}(X \xrightarrow {id_X} X)$ and cohomology operations}

Once again we emphasize that $\mathbb B^{coop}F^{*}(X \xrightarrow {\op{id}_X} X)$ consists of a collection of homomorphisms
$$c_g: F^m(X') \to F^{m+i}(X')$$
for all $g:X' \to X$  and for $X'' \xrightarrow {h} X' \xrightarrow g X$, the following diagram commutes:
$$
\xymatrix
{F^m(X'') \ar[d]_{c_{g \circ h}^k} && F^m(X')\ar[ll]_{(h')^*} \ar[d]^{c_g^k}\\
F^{m+i}(X'') && F^{m+i}(X') \ar[ll]^{h^*}.
}
$$ 
In other words, $c_g: F^m(X') \to F^{m+i}(X')$ is what is called \emph{a cohomology operation}, which is \emph{a natural self-transformation} of a contravariant functor. Therefore $ \mathbb B^{coop}F^{*}(X \xrightarrow {\op{id}_X} X)$ consists of natural self-transformations of the contravariant functor $F^*$. The above commutative cube (\ref{ccube-2}) means that $T$ is \emph{a natural transformation between two contravariant functors equipped with natural self - transformations}. This fact is sometimes expressed as ``cohomology operations of $F^*$ and $G^*$ are \emph{compatible with} the natural transformation $T: F^* \to G^*$", as seen above (e.g., see \cite{nLab}).
\begin{rem} Corresponding to the identity $\op{id}_X:X \to X$, $c_g: F^m(X') \to F^{m+i}(X')$ is a cohomology operation \emph{which is associated to the identity $\op{id}_X:X \to X$}. So, for a general map $f: X \to Y$,
$c_g: F^m(X') \to F^{m+i}(Y')$ could be called \emph{a ``generalized" cohomology operation\footnote{The referee pointed out that it is natural to consider cohomology operations that map between cohomologies of different spaces: power operation is such an example that occurs naturally (e.g., see \cite{Qui}).} associated to the map $f:X \to Y$}, where we consider
the fiber square \, \, \, 
\begin{equation*}
\CD
X' @> {g'}  >> X\\
@V {f'} VV @VV fV\\
Y' @>> {g} > Y.\\
\endCD
\end{equation*}
\end{rem}
\begin{rem}
As we remarked above, if $F^*(-)$ has values in an arbitrary category, then $ \mathbb B^{coop}F^{*}(X \xrightarrow {\op{id}_X} X)$ should be a collection of \emph{morphisms}, hence, for example, if we consider the category of rings, then it is a collection of ring homomorphisms; if we consider the category of sets, then it is a collection of maps. Even if we consider the category of rings or (graded) abelian groups, if we ignore some structures, e.g., ignore all the structures and consider only just sets, then it has to be a collection of maps. In other words, it depends on what kind of situations we consider. 
\end{rem}
\begin{rem} As pointed out by T. Annala (in private communication), if we consider a more refined co-operational bivariant theory $ \widehat{\mathbb B}^{coop}F^{*}$ requiring certain extra conditions, then we will obtain an isomorphism $\widehat{\mathbb B}^{coop}F^{*}(X \xrightarrow {id_X} X) \cong F^*(X)$, which is similar to the case of the refined operational bivariant theory of Chow homology theory treated in \cite[\S 17.3]{Ful}.
Details of this and some other refined co-operational bivariant theories will be treated in a different paper \cite{Yo2}.
\end{rem}

Here are some examples.
\begin{ex}\label{Adam} For this example, e.g., see Karoubi's book \cite[IV, 7.13 Theorem and V, 3.37 Theorem]{Kar}.
The $k$-th Adams operation 
$$\Psi^k:K(X) \to K(X)$$
 for the topological K-theory is a ring homomorphism and it turns out that the Adams operation is the only operation satisfying a ring homomorphism\footnote{If we consider only the additive structure, i.e., a group structure, ignoring the product structure, then there may be other operations.}. The Adams-like operation 
$$\Psi^k_H:H^{ev}(X;\bQ) \to H^{ev}(X;\bQ)$$
 for the even cohomology with rational coefficients is defined by
 $\Psi^k_H(x) := k^r x$ for $x \in H^{2r}(X, \bQ)$. The Chern character $ch:K(X) \to H^{ev}(X ;\bQ)$ is compatible with these Adams operations, i.e., the following diagram commutes:
$$
\xymatrix
{
K(X) \ar[d]_{\Psi^k}\ar[rr]^{ch}  && H^{ev}(X;\bQ) \ar[d]^{\Psi^k_H}\\
K(X) \ar[rr]_{ch} && H^{ev}(X;\bQ).
}
$$
which is compatible with pullback, i.e., for any continuous map $h:X \to Y$ we have the commutative cube:
\begin{equation*}
\xymatrix
{ K(Y) \ar[dd]_{\Psi^k } \ar[rd]^{h^*} \ar[rr]^{ch} && H^{ev}(Y;\bQ) \ar'[d][dd]^(.3){\Psi^k_H} \ar[rd]^{h^*} \\
& K(X) \ar[dd]_(.3){\Psi^k} \ar[rr]^(.3){ch}  &&  H^{ev}(X;\bQ)\ar[dd]^{\Psi^k_H} \\
K(Y)\ar'[r] [rr]_{ch \quad \quad }  \ar[rd]_{h^*} && H^{ev}(Y;\bQ)\ar[rd]_{h^*} \\
& K(X) \ar[rr] _{ch \quad \quad  \quad }  &&  H^{ev}(X;\bQ).
}
\end{equation*}
Then for $g:X' \to X$, we have the following commutative diagram:
$$
\xymatrix
{
K(X') \ar[d]_{\Psi^k}\ar[rr]^{ch}  && H^{ev}(X';\bQ) \ar[d]^{\Psi^k_H}\\
K(X') \ar[rr]_{ch} && H^{ev}(X';\bQ).
}
$$
That is, if we use the previous notation, for any $g:X' \to X$, we have $(\Psi ^k)_g = \Psi ^k$ and $(\Psi ^k)^{ch} = \Psi^k_H$.
 From the Chern character $ch:K(X) \to H^{ev}(X ;\bQ)$ we get the following Grothendieck transformation
$$\ga_{ch}: \mathbb B^{coop}_{ch} K^*(X \to Y)  \to  \mathbb B^{coop} H^{ev *}(X \to Y)_{\mathbb Q}.$$
Here we note that in this case $\mathbb B^{coop}_{ch} K^*(X \xrightarrow {\op{id}_X} X) = \mathbb B^{coop} K^*(X \xrightarrow {\op{id}_X} X)$, which consists of only Adams operations, provided that we require the homomorphism $c_g:K(X') \to K(X')$ to be a ring homomorphism. 
\end{ex}
\begin{ex} Let $\Omega^*(X)$ be Levine--Morel's algebraic cobordism on smooth schemes and let $pr: \Omega^*(X) \to CH^*(X)$ be the canonical map to the Chow cohomology defined by $pr([V \xrightarrow h X]):=h_*[V] \in CH^*(X)$. Let $LN_I: \Omega^*(X) \to \Omega^{* + |I|}(X)$ be the Landwever--Novikov operation for a partition $I = (n_1 \geq n_2 \geq \cdots \geq n_i>0)$ with $|I|:=n_1 + n_2 + \cdots + n_i$ being the degree of $I$, and $LN^{CH}_I: CH^*(X) \to CH^*(X)$ be the Chow cohomology version of the Landwever--Novikov operation (see \cite[\S 3]{Zain}). Then we have the following commutative diagram:
$$
\xymatrix
{
\Omega^*(X)  \ar[d]_{LN_I}\ar[rr]^{pr}  && CH^*(X) \ar[d]^{LN^{CH}_I}\\
\Omega^*(X)\ar[rr]_{pr} && CH^*(X).
}
$$

Similarly to the above, for $g:X' \to X$, we have the following commutative diagram:
$$
\xymatrix
{
\Omega^*(X')  \ar[d]_{LN_I}\ar[rr]^{pr}  && CH^*(X') \ar[d]^{LN^{CH}_I}\\
\Omega^*(X')\ar[rr]_{pr} && CH^*(X').
}
$$
That is, if we use the previous notation, for any $g:X' \to X$, we have $(LN_I)_g = LN_I$ and $(LN_I)^{ch} = LN^{CH}_I$.

From the canonical map $pr: \Omega^*(X) \to CH^*(X)$ we get the following Grothendieck transformation
$$\ga_{pr}: \mathbb B^{coop}_{pr} \Omega^*(X \to Y)  \to  \mathbb B^{coop} CH^*(X \to Y).$$
The Landwever--Novikov operations $LN_I: \Omega^*(X) \to \Omega^*(X)$ become elements of $\mathbb B^{coop}_{pr} \Omega^*(X \xrightarrow {id_X} X)$.  
\end{ex}
\begin{ex}\label{LN} Let $CH^*(X)/2$ be the Chow cohomology mod out by the 2-torsion part and
let $\overline{pr}:\Omega^*(X) \to CH^*(X)/2$ be the composition of the above $pr:\Omega^*(X) \to CH^*(X)$ and the projection $CH^*(X) \to CH^*(X)/2$. Then Brosnan, Levine and Merkurjev (see \cite{Vishik} and \cite{Bros}, \cite{Levine} and \cite{Merk} ) showed that there exists (unique) operation $S^i:CH^*(X)/2 \to CH^*(X)/2$, called Steenrod operation, making the following diagram commutative
$$
\xymatrix
{
\Omega^*(X)  \ar[d]_{S^i_{LN}}\ar[rr]^{\overline{pr}}  && CH^*(X)/2 \ar[d]^{S^i}\\
\Omega^*(X)\ar[rr]_{\overline{pr}} && CH^*(X)/2
}
$$
which is compatible with pullback. Here $S^i_{LN} := LN_I$ where  $I = (n_1 \geq n_2 \geq \cdots \geq n_i>0)$ with $n_1 =n_2 = \cdots =n_i=1$, thus $|I|=i$.
As above, we can get the following Grothendieck transformation
$$\ga_{\overline{pr}}: \mathbb B^{coop}_{\overline{pr}} \Omega^*(X \to Y)  \to  \mathbb B^{coop} CH^*(X \to Y)/2.$$
\end{ex}
\vspace{0.5cm}
\noindent
{\bf Acknowledgements}:

This work was motivated by a question posed by Toru Ohmoto  (in private communication) after his recent work \cite{Ohmoto} on Kazarian's conjecures \cite{Kaz} on Thom polynomials. The author thanks him for posing the question and some useful discussion. The author also thanks Burt Totaro for answering some questions about his work \cite{Totaro} and Toni Annala for some comments on the first draft of this paper.
Some of this work was done during the author's visiting Fakult\"at f\"ur Mathematik, Universit\"at Wien, in May 2023. The author thanks Herwig Hauser for a very nice hospitality and financial support. Finally the author thanks the anonymous referee for his/her careful reading of the paper and many valuable and constructive comments, suggestions and questions.
The author is supported by JSPS KAKENHI Grant Numbers JP19K03468 and JP23K03117.

\end{document}